\pdfoutput=1
\documentclass[final,3p,times]{elsarticle}

\usepackage{amssymb}
\usepackage{amsfonts}
\usepackage{color}
\usepackage{float}
\usepackage{alltt}
\usepackage{amsthm}
\usepackage{amsmath}
\usepackage{thmtools}
\usepackage{newtxtext,newtxmath}
\usepackage{enumerate}
\usepackage{enumitem}
\usepackage{footnote}
\usepackage{graphicx}
\usepackage{subcaption}
\usepackage{tabularx}
\usepackage{makecell}
\usepackage[colorlinks=true, allcolors=blue]{hyperref}
\usepackage[normalem]{ulem}

\newcommand{\N}{\mathbb{N}}
\newcommand{\R}{\mathbb{R}}
\newcommand{\C}{\mathbb{C}}

\DeclareMathOperator{\diag}{diag}

\newtheorem{theorem}{Theorem}[section]
\theoremstyle{definition} 

\newtheorem{lemma}[theorem]{Lemma}
\newtheorem{proposition}[theorem]{Proposition}
\newtheorem{remark}[theorem]{Remark}
\newtheorem{definition}[theorem]{Definition}
\newtheorem{corollary}[theorem]{Corollary}

\journal{}

\begin{document}

\begin{frontmatter}

\title{\textbf{On the stability of  second-order  methods for DDEs and PDDEs}}

\author[inst1]{Javier Rodr\'iguez-Guti\'errez}

\ead{javirodri@usal.es}

\affiliation[inst1]{
organization={Department of Applied Mathematics, University of Salamanca},
            city={Salamanca},
            country={Spain}}

\author[inst2]{Jesús Martín-Vaquero}

\ead{jesmarva@usal.es}

\affiliation[inst2]{
organization={Department of Applied Mathematics, IUFFyM, University of Salamanca},
            city={Salamanca},
            country={Spain}}

\begin{abstract}

In this paper, a one-parameter family of second-order implicit--explicit (IMEX) multistep methods is studied for delay differential equations (DDEs).  The analysis is based on the matrix problem $y'(t)=A y(t) + B y(t-\tau(t))$, where $A$ is Hermitian negative definite but $B$ is any matrix.  For constant, grid-aligned delays, a field-of-values argument
yields unconditional stability whenever a suitable generalized field
of values $F_p(A,B)$ is contained in the open unit disk. For a constant
non-grid-aligned delay, we show that quadratic interpolation does not
reduce the corresponding stability disk. For genuinely variable
delays, the method is formulated as a nonautonomous augmented system,
and product-level sufficient conditions are obtained from a common
quadratic Lyapunov inequality and from a quantitative slow-variation
hypothesis. 
 The approach developed herein is illustrated through several examples in which the discussed theory is applied not only to DDEs, but also to semidiscrete parabolic problems given by partial delay differential equations with a diffusion term.

\end{abstract}


\end{frontmatter}

\section{Introduction}
\label{sec:1}

Many dynamical processes do not react instantaneously to their present state. Transport, maturation, sensing, communication, and decision making introduce finite time lags, so that the rate of change at time $t$ depends on information generated at earlier times. Even a single delay may change the qualitative behaviour of a model by creating or suppressing oscillations, moving stability boundaries, or producing bifurcations that are absent from the corresponding ordinary differential equation. Delay differential equations (DDEs) therefore provide a natural framework for feedback systems with memory, with applications ranging from population dynamics and epidemiology to physiology, chemical kinetics, control, and networked engineering systems; see, for example, \cite{HaleVerduynLunel1993,HartungKrisztinWaltherWu2006}. State-dependent delays are especially relevant when the travel, processing, or maturation time is itself affected by the evolving state, and they lead to a substantially richer dynamical theory than constant-delay models.

A DDE initial-value problem is infinite-dimensional, since its solution is determined by an initial history function rather than by a single vector. This distinction has direct numerical consequences. A time-stepping method must repeatedly evaluate an already computed approximation at delayed arguments, track propagated discontinuities, and provide a continuous extension or interpolation whenever a delayed point does not coincide with the mesh. These issues are central to the convergence and implementation theory of DDE solvers \cite{bellen-zennaro-DDE}. They become more delicate for time- or state-dependent delays: the delayed point moves relative to the grid, may cross mesh intervals as the solution evolves, and can make an otherwise explicit formula effectively implicit. Practical solvers for such problems consequently rely on continuous Runge--Kutta extensions, event handling, and residual control \cite{Shampine2005}.

In this work we consider the matrix problem
\begin{equation}
  y'(t)=A y(t)+B y\!\left(t-\tau(t,y(t))\right),
  \label{eq:intro-matrix-dde}
\end{equation}
where $A$ and $B$ are not assumed to be simultaneously diagonalizable. Although the dependence on the present and delayed states is linear, a state-dependent argument makes the complete functional equation nonlinear; the linear matrix structure nevertheless remains the natural core for local, frozen-delay, and numerical stability analyses. For a uniform mesh with step size $h$, let $\tau_n=\tau(t_n,y(t_n))$ and write
\[
  \tau_n=(m_n+u_n)h,
  \qquad m_n=\left\lfloor\frac{\tau_n}{h}\right\rfloor,
  \quad 0\leq u_n<1.
\]
Under the no-overlap assumption $\tau_n\ge h$, we have $m_n\ge1$ and
$t_n-\tau_n=t_{n-m_n}-u_n h$.  The classical grid-aligned setting corresponds to $u_n=0$, whereas a variable delay generally produces changing integers $m_n$ and nonzero fractional offsets $u_n$. Thus, interpolation is not merely an implementation detail: its weights enter the discrete recurrence and can alter the numerical stability mechanism.

Most of the classical numerical stability theory was developed from the scalar constant-delay test equation
\begin{equation}
  y'(t)=\lambda y(t)+\gamma y(t-\tau),
  \qquad \lambda,\gamma \in\mathbb{C},\quad \tau>0,
  \label{eq:intro-scalar-test}
\end{equation}
usually together with the constrained-mesh relation $\tau=mh$. Early work introduced delay analogues of $A$-stability and studied highly stable multistep formulas, backward differentiation methods, and integrator--interpolator pairs \cite{Cryer1974,Barwell1975,Bickart1982}. 

It is useful to distinguish two independent classifications used in this literature: delay-independent versus delay-dependent stability describes whether a stability statement is uniform with respect to the delay or depends explicitly on it, while unconditional versus conditional numerical stability describes whether stability holds for every admissible $h>0$ or only under a step-size restriction.

For one-step schemes, a substantial theory is available. P-stability of natural Runge--Kutta processes on constrained meshes was investigated in \cite{Zennaro1986}; necessary or sufficient step-size conditions for $\theta$-methods were obtained in \cite{CalvoGrande1989}; and a general framework for step-by-step methods with complex coefficients was developed in \cite{InTHoutSpijker1991}. Since the delayed value is normally off-grid, the choice of continuous extension is inseparable from the stability question, as illustrated by the interpolation procedure of \cite{InTHout1992}. System-level results under suitable matrix or operator assumptions were established for natural Runge--Kutta methods in \cite{Koto1994}. Later studies refined the distinction between delay-independent and delay-dependent regions for Runge--Kutta methods \cite{Guglielmi1997} and identified important barriers to unconditional stability for broad classes of implicit schemes \cite{GuglielmiHairer2001}. Nevertheless, much of this theory remains scalar, assumes a constant delay, or reaches systems only through structural hypotheses that reduce the analysis to scalar modes.

The multistep literature follows a similar progression. Delay-dependent stability was analysed for fixed and distributed delays in \cite{HuangVandewalle2004}, while scalar stability regions for implicit--explicit linear multistep formulas were studied in \cite{Koto-Stability}. General-linear-method criteria were subsequently specialized to linear multistep and multistep--multistage schemes in \cite{HuangHuTian2009}. Of particular relevance to non-grid-aligned delays, Huang and Vandewalle proved unconditional stability for specific fully discrete parabolic delay problems on nonconstrained meshes, using linear interpolation together with the trapezoidal rule or BDF2 \cite{HuangVandewalle2012}. These results demonstrate that a fractional delay offset can be incorporated without destroying stability, but they do not provide a general multistep theory for a noncommuting matrix pair $(A,B)$ when the offset varies from step to step.

Nonconstant delays have also been treated from several complementary viewpoints. Zennaro studied asymptotic contractivity for continuous Runge--Kutta methods applied to nonlinear DDEs with time-dependent delays \cite{Zennaro1997}. Zhao, Liu, and Qiu analysed three $\theta$-type discretizations for equations with continuous matrix coefficients and a continuous variable lag \cite{ZhaoLiuQiu1999}. These contributions are fundamental, but their principal aims are contractivity, particular one-step formulas, special variable-delay structures, or reliable error control. For the two-step family studied here, it remains useful to
combine noncommuting matrix coefficients, off-grid quadratic
reconstruction, and product-level conditions for a genuinely variable
delay within one analysis.

Recent work has revisited constant-delay matrix equations without a simultaneous-diagonalization hypothesis. Field-of-values arguments have yielded conditional and unconditional stability conditions for $\theta$-methods \cite{Art-Alejandro} and for IMEX--BDF formulas \cite{articulo_ana}. The latter analysis introduces a fractional-offset representation, but subsequently specializes to the aligned case $u=0$. The present work extends this matrix-level program to a family of second-order multistep methods while retaining
the fractional position of the delayed argument.

Three different stability regimes must be distinguished. First, for a
constant grid-aligned delay, we derive an unconditional stability result
by combining scalar boundary-locus analysis with a generalized
field-of-values argument. Second, for a constant non-grid-aligned delay,
we prove that quadratic interpolation does not reduce the scalar
stability disk. Third, for a genuinely time- or
state-dependent delay, the integer delay index and fractional offset
may change from one step to the next. The resulting recurrence is
nonautonomous, and stability is governed by products of step matrices
rather than by the eigenvalues of each frozen matrix separately. Thus,
frozen-delay stability must be supplemented by a product-level
condition, such as a common quadratic Lyapunov inequality.

Therefore, the outline of the paper is as follows: in \S \ref{sec:2} the general framework for consistency and stability of multistep methods is introduced in the scalar and matrix cases, using the concept of field of values and its properties. 
In \S \ref{s3:new-family}  the
new second-order family is derived, and general conditions to obtain unconditional stability  when $u=0$ are obtained.

  Later, Section~\ref{s4:non-constant-delay} first considers constant
non-grid-aligned delays and then discusses genuinely variable and
state-dependent delays. The examples in that section verify the frozen
stability hypotheses and illustrate second-order convergence.   
    Finally, in \S \ref{sec:numer-simul}, some numerical simulations in DDEs and PDDEs give a good idea of the power of the theory developed.

\section{Theoretical Framework and Stability Analysis }
\label{sec:2}
First of all, consider a delay differential equation (DDE) with the form:
\begin{equation}
    \left\{\begin{array}{lcc}
            y'(t)=f(t,y(t))+g(t,y(t),y(t-\tau)), & t \ge t_0, \\
            y(t)=\phi(t), & t\le t_0,
        \end{array}
        \right.
        \label{DDE_general}
\end{equation}
where $y:[t_0-\tau,\infty) \rightarrow \mathbb{R}^d$ is a $\mathcal{C}^3([t_0-\tau,\infty))$ function called the state variable, $f:[t_0,\infty) \times \mathbb{R}^d \rightarrow \mathbb{R}^d$, $g:[t_0,\infty) \times \mathbb{R}^d  \times \mathbb{R}^d \rightarrow \mathbb{R}^d$ and $\phi:[t_0-\tau,t_0]\rightarrow \mathbb{R}^d$  is the prescribed history.  Whenever an order statement is made, the solution and the data are
assumed to be sufficiently smooth and to satisfy the required
compatibility conditions at breaking points.

For the moment, assume that $\tau=mh$ with $m\in\N$, $m\ge1$, so that
all delayed arguments lie on the uniform mesh $t_n=t_0+nh$.  We define a k-step linear multistep implicit-explicit (IMEX) method (with $k\ge 1$) for \eqref{DDE_general} as
\begin{equation}
    \sum\limits_{j=0}^k \alpha_j y_{n+j}=h \sum\limits_{j=0}^k \beta_j f(t_{n+j},y_{n+j})+h \sum\limits_{j=0}^{k}\beta^*_j g(t_{n+j},y_{n+j},y_{n-m+j}).
    \label{IMEX_general}
\end{equation}
Here $h=\tau/m$, with $m\in\mathbb{N}$ and $m\ge1$; $y_j\approx y(t_j)=y(t_0+jh)$ denotes the approximation of the state variable in the j-th step; and $\{\alpha_j,\beta_j,\beta_j^*\}_j$ is the set of coefficients of the method. The following notation will be used throughout the paper: $f_i=f(t_i,y_i)\approx f(t_0+ih,y(t_0+ih))$, $g_i=g(t_i,y_i,y_{i-m})\approx g(t_0+ih,y(t_0+ih),y(t_0+(i-m)h))$.

The aforementioned coefficients $\alpha_j,\beta_j$ will be determined via imposing convergence conditions on the implicit method (without the delay term), whereas the $\beta_j^*$ will be deduced by symmetry among coefficients.

\begin{definition}
 The IMEX linear multistep method \eqref{IMEX_general} is said to be consistent if the local truncation error, whose definition is
 \begin{equation}
     \begin{aligned}
     (T_h)_n={}&\frac{1}{h}\sum_{j=0}^k \alpha_j y(t_{n+j})
     -\sum_{j=0}^k \beta_j f(t_{n+j},y(t_{n+j}))\\
     &-\sum_{j=0}^k\beta_j^*
     g(t_{n+j},y(t_{n+j}),y(t_{n+j-m}))
     \end{aligned}
     \label{local_truncation_error}
 \end{equation}
 where $y(t_{n+j})$ represents the exact solution of \eqref{DDE_general} at time $t_{n+j}$, satisfies
 \begin{equation}
     \|T_h\|_\infty \xrightarrow[]{h\rightarrow 0}0.
 \end{equation}
\end{definition}

\begin{definition}
    The linear multistep IMEX method \eqref{IMEX_general} is consistent of order $p$ if
    \begin{equation}
        \|T_h\|_\infty=O(h^p).
    \end{equation}
\end{definition}

\subsection{Linear stability in a scalar DDE}

To motivate the choice of the $\beta_j^*$ coefficients, let us consider the simplest case of scalar DDE:
\begin{equation}
    \left\{
    \begin{array}{lcc}
       y'(t)=\lambda y(t)+\gamma y(t-\tau)=\lambda(y(t)-\mu y(t-\tau))  & t\ge t_0 \\
       y(t)=\phi(t)  & t\le t_0,
    \end{array}
    \right.
    \label{Scalar_DDE}
\end{equation}
where $\lambda <0$ is a real number, $|\gamma|< - \lambda$, $\mu=-\frac{\gamma}{\lambda}$ and $\tau=mh > 0$ is a scalar. Applying the described method \eqref{IMEX_general} to this problem leads us to

\begin{equation}
    \sum\limits_{j=0}^k \alpha_j y_{n+j}=z\sum\limits_{j=0}^k \beta_j y_{n+j}+w\sum\limits_{j=0}^k \beta_j^* y_{n+j-m},
    \label{method_scalar_case}
\end{equation}
having defined $z=\lambda h$ and $w=\gamma h=-z\mu$.

The stability analysis in this case just involves controlling the roots of the characteristic equation. We define the polynomials

\begin{equation}
    \rho(\zeta)=\sum\limits_{j=0}^k \alpha_j \zeta^j, \quad \sigma(\zeta)=\sum\limits_{j=0}^k \beta_j \zeta^j, \quad \sigma^*(\zeta)=\sum\limits_{j=0}^k \beta_j^* \zeta^j.
    \label{polynomials_characteristic_equation}
\end{equation}

For the stiff limit, $z=-\infty$ is not substituted directly into
\eqref{char_equation}. Instead, after division by $-z$, we define
\begin{equation}
  P_{m,\infty}(\mu,\zeta)
  :=\zeta^m\sigma(\zeta)-\mu\sigma^*(\zeta).
  \label{eq:stiff-limit-characteristic}
\end{equation}

The characteristic equation turns out to be
\begin{equation}
    P(\mu,z,\zeta)=\zeta^m (\rho(\zeta)-z\sigma(\zeta))+z\mu \sigma^*(\zeta)=0.
    \label{char_equation}
\end{equation}

Throughout, $D(c,r):=\{w\in\C:|w-c|<r\}$ denotes an open disk.

\begin{definition}
 For fixed $m\ge1$, $z\in \mathbb{R}_{<0}\cup\{-\infty\}$, and
 $\mu\in\mathbb C$,   the characteristic equation \eqref{char_equation} is said to be stable if every root satisfies $|\zeta|<1$ for that particular $z\in \mathbb{R}_{<0}\cup\{-\infty\}$ (using
 \eqref{eq:stiff-limit-characteristic} when $z=-\infty$).
\end{definition}

\begin{definition}
    The P-stability region of \eqref{method_scalar_case} is defined as
    \begin{equation*}
        D_P=\bigcap\limits_{m\ge 1}D_P^m,\quad D_P^m=\{(z,\mu)\in \mathbb{C}^2 \;/\; \eqref{char_equation} \text{ is stable}\}.
    \end{equation*}
\end{definition}
\begin{definition}
    For fixed $m\in \mathbb{N}$ and $z\in \mathbb{R}_{<0}\cup\{-\infty\}$, we define the stability region $D_z$ as
    \begin{equation*}
        D_z=\{\mu\in \mathbb{C}\;/\; \eqref{char_equation} \text{ is stable}\},
    \end{equation*}
    and its boundary will be denoted by
    \begin{equation}
        \Gamma_z=\left\{\mu=-\dfrac{\zeta^m(\rho(\zeta)-z\sigma(\zeta))}{z\sigma^*(\zeta)}\in \mathbb{C}\;/\;\zeta=e^{i\theta},0\le \theta \le 2\pi\right\}.
        \label{Gamma_z}
    \end{equation}
\end{definition}

\begin{definition}
    We denote the infimum of the modulus of all the points in $\Gamma_z$ by $\sigma_z$.
\end{definition}

\begin{definition}
    Given the Cauchy problem
    \begin{equation}
             \left\{
        \begin{array}{lcc}
             y'(t)=\lambda y(t),& t\ge t_0, \\
             y(t_0)=y_0, 
        \end{array}
        \right.
        \label{A-stability_problem}
    \end{equation}
    being $\lambda \in \mathbb{C}$, and set $z=h\lambda$. If a method applied to the problem can be expressed as $y_{n+1}=R(h_{n+1}\lambda)y_n$, with $R(z)$ the so-called stability function, the absolute stability region for \eqref{A-stability_problem} is defined as
    \begin{equation}
        S_A=\{z\in\mathbb{C}\;/\; |R(z)|<1\}.
    \end{equation}
    The formula is $A$-stable if
    $\{z\in\mathbb C:\Re z\le0\}\subseteq S_A$.
\end{definition}

If we denote by $D(0,\sigma_{z_j})$ the disk centered at the origin with radius $\sigma_{z_j}\in \mathbb{R}_{>0}\cup \{\infty\}$, and the method \eqref{method_scalar_case} verifies $[-\infty,0)\subset S_A$, then $D(0,\sigma_{z_j})\subseteq D_{z_j}$ $\forall z_j \in \mathbb{R}_{<0}\cup \{-\infty\}$.

The following Lemma allows us to characterize the P-stability region:
\begin{lemma}[\cite{Koto-Stability}, Theorem 1]
    Suppose the vectors formed by the coefficients related to the implicit part of \eqref{method_scalar_case}, $(\alpha_0,\dots,\alpha_k)$ and $(\beta_0,\dots,\beta_k)$, are linearly independent. Consider the following statements:
    \begin{enumerate}[label=(\alph*)]
  \item $z\in S_A$, $|\mu|<\sigma_z$. 
  \item $(z,\mu)\in D_P$.
  \item $z\in S_A$, $|\mu|\le \sigma_z$.
\end{enumerate}
Then (a) $\Rightarrow$ (b) $\Rightarrow$ (c).
\end{lemma}

\begin{corollary}
    Given an $A$-stable method with a stable stiff limit, for any $z\in \mathbb{R}_{<0}\cup \{-\infty\}$ we have $D(0,\sigma_z)\subseteq D_z$, with $D(0,\sigma_z)$ being the disk centered in the origin with radius $\sigma_z$.
\end{corollary}

\begin{definition}
    We will call the region of unconditional stability to
    \begin{equation}
        D=\bigcap\limits_{z\in\mathbb{R}_{<0}\cup\{-\infty\}}\bigcap\limits_{m\ge 1} D_z,
         \label{eq:unconditional-region}
    \end{equation}
    which is the set of $\mu\in \mathbb{C}$ such that \eqref{method_scalar_case} is stable for every $z\in\mathbb R_{<0}\cup\{-\infty\}$ and every $m\ge1$.
\end{definition}

\begin{remark}
    Our focus is to deal with problems that are asymptotically stable, which in the scalar case translates to $|\mu|<1$ (see \cite[Chapter~9]{bellen-zennaro-DDE}, Chapter 9). Hence our goal is to design a method verifying $D(0,1)\subseteq D$, this is, an unconditionally stable method.
\end{remark}

\subsection{Linear stability in the general case}
Consider we have the matrix-based problem
\begin{equation}
    \left\{
    \begin{aligned}
         y'(t)=&Ay(t)+By(t-\tau), & t\ge& t_0, \\
         y(t)=&\phi(t), & t\le& t_0,
    \end{aligned}
    \right.
    \label{prob_matricial_general}
\end{equation}
with $A,B\in \mathbb{C}^{d \times d}$. Assume $A$ is Hermitian negative definite, that is, it coincides with its conjugate transpose $A^*=A$ and $\langle x,Ax\rangle<0$ $\forall x\not= 0$, $x\in \mathbb{C}^d$.

\begin{remark}
    In the particular case in which $A$ and $B$ are simultaneously diagonalizable, i.e., $A=Q\Lambda Q^{-1}$, $B=Q\Gamma Q^{-1}$ with $\Lambda=\diag(\lambda_i)$, $\Gamma=\diag(\gamma_i)$, the diagonal matrices whose non-zero elements are the eigenvalues of A and B, respectively.
    
    Taking $y=Qu$, the problem can be rewritten as
    \begin{equation}
        u'(t)=\Lambda u(t)+\Gamma u(t-\tau),
    \end{equation}
    and therefore can be decoupled into a system of d independent equations $u_i'(t)=\lambda_i u_i(t)+\gamma_i u_i(t-\tau)$.
    
    Finally, we can apply the theory developed for the scalar case \eqref{Scalar_DDE} to each equation.
\end{remark}

Let us assume, then, that $A$ and $B$ do not diagonalize simultaneously. We will make use of the concept called the \textit{field of values} of a matrix and develop a procedure to obtain stability conditions for \eqref{prob_matricial_general} (see \cite{LMM-Seib-Ros} for its application to a non-delayed IMEX scheme).
\begin{definition}
    We define the field of values (FOV) of a matrix $X\in \mathbb{C}^{d\times d}$ as
    \begin{equation*}
        F(X):=\left\{ \langle x,Xx\rangle : ||x||^2=1,\; x\in \mathbb{C}^d \right\}.
    \end{equation*}
\end{definition}

For a fixed $p\in\R$, let us define the generalized field of values as:
\begin{equation*}
    F_p\equiv F_p(A,B):=\left\{ \langle v,(-A)^{p-1}Bv\rangle:\langle v,(-A)^p v\rangle=1,\; v\in \mathbb{C}^d, \; p\in \mathbb{R} \right\}.
\end{equation*}

Via a change of variable $v=(-A)^{-\frac{p}{2}}x$, we have that
\begin{equation}
    \begin{aligned}
        &\langle v,(-A)^{p-1}Bv \rangle = \langle (-A)^{-\frac{p}{2}}x,(-A)^{p-1}B(-A)^{-\frac{p}{2}}x\rangle=((-A)^{-\frac{p}{2}}x)^*(-A)^{p-1}B(-A)^{-\frac{p}{2}}x
        \\&=x^* ((-A)^{-\frac{p}{2}})^* (-A)^{p-1}B (-A)^{-\frac{p}{2}}x\overset{\begin{array}{cc}
             \footnotesize\text{A Hermitian}  \\
             \downarrow 
        \end{array}}{=} x^* (-A)^{-\frac{p}{2}} (-A)^{p-1}B(-A)^{-\frac{p}{2}}x
        \\&=x^* (-A)^{\frac{p}{2}-1}B(-A)^{-\frac{p}{2}}x=\langle x,(-A)^{\frac{p}{2}-1}B(-A)^{-\frac{p}{2}}x \rangle,
    \end{aligned}
    \label{Transformacion_v_en_x}
\end{equation}
thus, it is possible to assert the following identification
\begin{equation*}
    F_p=F((-A)^{\frac{p}{2}-1}B(-A)^{-\frac{p}{2}})=\left\{ \langle x,(-A)^{\frac{p}{2}-1}B(-A)^{-\frac{p}{2}}x \rangle \in \mathbb{C} :\; x\in \mathbb{C}^d,\; ||x||=1 \right\}.
\end{equation*}

Let us state some analytic and geometric properties of the FOV, which will be used later:

\begin{remark}
For $d\times d$ matrices, the properties and definitions used below are:
\begin{enumerate}[label=(\alph*)]
    \item $F(X)$ is convex, closed and bounded.
    \item $\sigma(X) \subseteq F(X)$.
    \item $F(X_1)+F(X_2)=\{ x_1+x_2:\; x_1 \in F(X_1),x_2 \in F(X_2) \}$.
    \item $F(X_1)F(X_2)=\{ x_1 x_2 : x_1 \in F(X_1),x_2 \in F(X_2) \}$.
    \item $F(w_1 X_1+w_2 I)=w_1 F(X_1)+w_2 $.
    \item If $X_1$ is positive semi-definite, then $\sigma(X_1 X_2)\subseteq F(X_1)F(X_2).$
    \item If $X$ is normal, $F(X) = Co(\sigma(X))$, where $Co(\cdot)$ denote the closed convex hull of a set.
    \item If $X$ is Hermitian, $F(X)$ is a segment of the real line.
\end{enumerate}
\end{remark}

Their proofs, as well as other relations, can be found in \cite{Horn-fieldofvalues}.

Before connecting all this with our scheme, let us introduce two more concepts regarding the "size" of both the field of values and the spectrum of a matrix.
\begin{definition}
    The numerical radius of a matrix $X\in \mathbb{C}^{d\times d}$ is
    \begin{equation*}
        r(X)=\max\limits_{c\in F(X)}|c|.
    \end{equation*}
\end{definition}

As a consequence, by definition the field of values of a matrix $X$ is contained within the closed disk centered at the origin of radius $r(X)$, i.e., $F(X)\subseteq D(0,r(X))$.
\begin{definition}
    We will call spectral radius of a matrix $X\in \mathbb{C}^{d\times d}$ to
    \begin{equation*}
        \rho(X)=\max\limits_{\lambda\in\sigma(X)}|\lambda|.
    \end{equation*}
\end{definition}

On this basis, by applying the method \eqref{IMEX_general} to \eqref{prob_matricial_general}, we end up with
\begin{equation}
    \sum\limits_{j=0}^k \alpha_j y_{n+j}=h\sum\limits_{j=0}^k \beta_j A y_{n+j}+h\sum\limits_{j=0}^{k}\beta_j^*B y_{n+j-m}.
    \label{Metodo.retardo.IMEX.matricial}
\end{equation}

Since $M:= \alpha_kI-h\beta_kA$ is nonsingular and following the reasoning used for the non-delayed case in \cite{LMM-Seib-Ros}, we will transform this linear recursive relation into a vectorial one:
\begin{equation*}
    Y_n=WY_{n-1},\quad \text{with }Y_n:=(y_{n+k},y_{n+k-1},\dots,y_{n-m+1})^T \in \mathbb{R}^{(k+m)\times d},
\end{equation*}
having assumed that $m\ge 3$.

The $(k+m)d\times (k+m) d$ matrix W is determined by
\begin{equation}
    W=\begin{pmatrix}
        M^{-1}C_1&M^{-1}C_2&\cdots&M^{-1}C_{k+m}\\
        I&0&\cdots&0\\
        &\ddots&\ddots&\vdots\\
        0&&I&0
    \end{pmatrix}.
    \label{matrix.W}
\end{equation}
Here $I=I_d$ and, for $r=1,\dots,k+m$,
\begin{equation*}
    C_r=h\beta_{k-r}A-\alpha_{k-r}I
        +h\beta^*_{k+m-r}B,
\end{equation*}
where $\alpha_j=\beta_j=\beta_j^*=0$ whenever
$j\notin\{0,\dots,k\}$. This convention also covers any overlap of
the current and delayed blocks.

Let V be an eigenvector of W with eigenvalue $\zeta$. Then, by observing the shape of W, for some $v\not= 0$, $v\in \mathbb{C}^d$, the eigenvector can be written as
\begin{equation*}
    V=\big(\zeta^{k+m-1}v,\zeta^{k+m-2}v,\dots,\zeta v,v\big)^T \in \mathbb{C}^{(k+m)d}.
\end{equation*}

In this sense, the characteristic equation of W can be expressed as
\begin{equation*}
    \det(W-\zeta I)=0 \iff \det\left( \zeta^m \Big( \frac{1}{h}\rho(\zeta)I-\sigma(\zeta)A \Big)-\sigma^*(\zeta)B \right)=0, 
\end{equation*}
here $\rho(\zeta),\sigma(\zeta)$ and $\sigma^*(\zeta)$ represent the polynomials \eqref{polynomials_characteristic_equation}.

Hence, for each eigenvalue $\zeta$ of W, there will exist an associated eigenvector V with the previous form, and whose $v\in\mathbb{C}^d$ will verify:
\begin{equation}
    T(\zeta)v=0,\quad \text{with } \; T(x)=x^m \left( \frac{1}{h}\rho(x)I-\sigma(x)A \right)-\sigma^*(x)B.
    \label{Ec_T(x)}
\end{equation}

The next result allows us to  characterize the stability of the method \eqref{Metodo.retardo.IMEX.matricial} via the characteristic equation of W:
\begin{proposition}[\cite{LMM-Seib-Ros}, Proposition 2]
    If the matrix $T(x)$ is nonsingular for all $|x|\ge 1$, then the method \eqref{Metodo.retardo.IMEX.matricial} is stable.
\end{proposition}

This criterion still has one main disadvantage, the matrix $T(x)$ relates the coefficients $\alpha_j,\beta_j,\beta_j^*$ with the matrices $hA,hB$. In order to decouple them, we multiply by the matrix $(-A)^{p-1}$, which is positive definite, on the left side of $T(\zeta)v=0$,   obtaining:
\begin{equation*}
    \frac{1}{h}\zeta^m \rho(\zeta)(-A)^{p-1}v=-\zeta^m \sigma(\zeta)(-A)^p v+\sigma^*(\zeta)(-A)^{p-1}Bv.
\end{equation*}

Now multiplying on the left by $\bar{v}$, using that $\langle v_1,v_2\rangle=\bar{v_1}v_2$ in $\mathbb{C}^d$, and defining the following parameters:
\begin{equation}
    z=-h\dfrac{\langle v,(-A)^pv\rangle}{\langle v,(-A)^{p-1}v\rangle},\quad \mu=\dfrac{\langle v, (-A)^{p-1}Bv\rangle}{\langle v, (-A)^p v\rangle},
    \label{def_z_mu}
\end{equation}
we arrive to the equation
\begin{equation}
    \zeta^m \rho(\zeta)=\zeta^m z\sigma(\zeta)-z\mu \sigma^*(\zeta).
    \label{char_equation_FOV}
\end{equation}

If $\langle v,A^p v\rangle =1$, then $\mu \in F_p$. If $\langle v,A^p v\rangle =1$, consider the normalized vector $\hat{v}=\frac{v}{\sqrt{\langle v, (-A)^p v \rangle}}$, where the denominator is well defined as (-A) is positive definite. Given these considerations, we have that
\begin{equation*}
    \langle \hat{v},(-A)^p \hat{v} \rangle = \Big\langle \dfrac{v}{\sqrt{\langle v,(-A)^p v \rangle}},(-A)^p \dfrac{v}{\sqrt{\langle v,(-A)^p v \rangle}} \Big\rangle=\dfrac{\langle v,(-A)^p v \rangle}{\langle v,(-A)^p v \rangle}=1,
\end{equation*}
\begin{equation*}
    \langle \hat{v},(-A)^{p-1}B \hat{v} \rangle = \Big\langle \dfrac{v}{\sqrt{\langle v,(-A)^p v \rangle}},(-A)^{p-1}B \dfrac{v}{\sqrt{\langle v,(-A)^p v \rangle}} \Big\rangle=\dfrac{\langle v,(-A)^{p-1}B v \rangle}{\langle v,(-A)^p v \rangle}=\mu,
\end{equation*}
and therefore $\mu=\big\langle \hat{v},(-A)^{p-1}B\hat{v}\big\rangle\in F_p$.

Let us observe that $z\in F(hA)$: suppose $\langle v,(-A)^{p-1} v\rangle =1$; if it is not the case, repeat the normalization $\hat{v}=\frac{v}{\sqrt{\langle v, (-A)^p v \rangle}}$. Now, using the change of variable $v=(-A)^{-\frac{p}{2}+\frac{1}{2}}$:
\begin{equation*}
\begin{aligned}
    \langle &v,(-A)^p v\rangle=\langle (-A)^{\frac{1-p}{2}}x,(-A)^p(-A)^{\frac{1-p}{2}}x\rangle=x^*\left( (-A)^{\frac{1-p}{2}} \right)^* (-A)^p(-A)^{\frac{1-p}{2}}x=\\
    &\;\;\quad\;\underset{\begin{array}{cc}
         \;\:\:\uparrow&  \\
         \:\;\scriptstyle A \text{ Hermitian}& 
    \end{array}}{=}x^*(-A)^{\frac{1-p}{2}} (-A)^p(-A)^{\frac{1-p}{2}}x=x^* (-A)x=\langle x,(-A)x\rangle,
    \end{aligned}
\end{equation*}
which allows us to conclude with $z\in F(hA)$.

\begin{remark}
    The equation \eqref{char_equation_FOV} has the same functional form as the characteristic equation \eqref{char_equation} derived from applying the IMEX scheme \eqref{IMEX_general} to the scalar problem \eqref{Scalar_DDE}. Therefore, we will treat the equation in the same sense as explained before. 
\end{remark}

In regard to the asymptotic stability of the problem \eqref{prob_matricial_general}, we have the following results given by K. J. in' t Hout:

\begin{theorem}[\cite{inthout1994stability}]
Consider the problem \eqref{prob_matricial_general}, and define $H^*=\{\zeta \;\; | \;\; \zeta \in \mathbb{C}, \;\; \Re(\zeta)<0 \}$. Let $\sigma(A), \rho(A)$  be the spectrum and spectral radius, respectively, of a matrix A, then the following conditions are equivalent:

\begin{align}
    (1) &\quad \begin{aligned}
        &\text{The exact solution } y(t) \text{ of } \eqref{prob_matricial_general} \text{ satisfies } \lim\limits_{t\rightarrow\infty}y(t)=0 \text{ (for } \tau>0 
        \text{), }
        \\
        &
        \text{i.e it is asymptotically stable.}
    \end{aligned} \\
    (2) &\quad \begin{aligned}
        &\sigma(A) \subset H^*, \;\; \rho((\zeta I-A)^{-1}B)<1 \text{ (as long as } \Re(\zeta)=0, \\
        & \zeta \neq 0 \text{)}, \text{ and } -1 \notin \sigma(A^{-1}B).
        \label{cond_fov_asint_estab}
    \end{aligned}
\end{align}
\end{theorem}

Until now, we have assumed that the matrices A and B are such that A is Hermitian and negative definite. Furthermore, let us suppose that $A^{-1}B$ is normal, i.e., $(A^{-1}B)^*A^{-1}B=A^{-1}B(A^{-1}B)^*$, where $(A^{-1}B)^*$ denotes the conjugate transpose. Under this condition, we can deduce that (see \cite{HornJohnson2013MatrixAnalysis}):
\begin{equation*}
    r(A^{-1}B)=\rho(A^{-1}B),
\end{equation*}
and that the following inequality holds
\begin{equation*}
    \rho(A^{-1}B)\le \sup\limits_{\Re{(\zeta)=0}}\rho((\zeta I-A)^{-1}B).
\end{equation*}

Finally,  we arrive at \begin{equation*}
    F(A^{-1}B) \subseteq D(0,r(A^{-1}B)) = D(0,\rho(A^{-1}B)) \subseteq D(0,1).
\end{equation*}

Hence, summing up all the above, we end up with:
\begin{corollary}
\label{corol_FOV_stability}
    Suppose the problem \eqref{prob_matricial_general} is asymptotically stable. Then if $A^{-1}B$ is a normal matrix, any method whose unconditional stability region contains $D(0,1)$, applied to \eqref{prob_matricial_general}, will be unconditionally stable.
\end{corollary}


\section{Designing the family of methods}
\label{s3:new-family}

Consider the general formula for 2-step linear methods:
\begin{equation}
\begin{aligned}
    \alpha_2 y_{n+2}+&\alpha_1 y_{n+1}+\alpha_0 y_n=h(\beta_2 f(t_{n+2},y_{n+2})+\beta_1 f(t_{n+1},y_{n+1})+\beta_0 f(t_n,y_n)\\
    & +\beta_2^* g(t_{n+2-m},y_{n+2},y_{n+2-m})+\beta_1^* g(t_{n+1-m},y_{n+1},y_{n+1-m})+\beta_0^* g(t_{n-m},y_{n},y_{n-m})).
    \end{aligned}
\end{equation}
The choice of coefficients that we will make comes from the following arguments: we normalize
\(\alpha_2=1\), choose \(\beta_2=1\), and impose the second-order
conditions; 0-stability and A-stability.
\newline
Moreover, if we apply the method to the scalar problem \eqref{Scalar_DDE} and then take the limit $z \rightarrow -\infty$ in the expression of $\mu\in \Gamma_z$:
\begin{equation}
    \mu_{m,z,\theta}=-\dfrac{(e^{i\theta})^m(\rho(e^{i\theta})-z\sigma(e^{i\theta}))}{z\sigma^*(e^{i\theta})}\xrightarrow[]{z\rightarrow-\infty}e^{im\theta}\dfrac{\beta_0+\beta_1 e^{i\theta}+\beta_2 e^{i2\theta}}{\beta_0^*+\beta_1^* e^{i\theta}+\beta_2^* e^{i2\theta}},
\end{equation}
we can directly observe a symmetry in the coefficients, so by setting $\beta_j=\beta_j^*$ we could assure that $|\mu_{-\infty}(\theta)|=1$ $\forall \theta$, hence $\Gamma_{-\infty}=\partial D(0,1)$. 
\newline
Proving that $\Gamma_z$ is a closed curve $\forall z\in \mathbb{R}_{<0}\cup \{-\infty\}$, and that its innermost region coincides with $D_z$, would imply that this method is stable for $z=-\infty$.
\newline
Finally, we end up with the following family of methods:
\begin{equation}
    y_{n+2}+2(\beta_0-1) y_{n+1}+(1-2\beta_0) y_n=z( y_{n+2}+(\beta_0-1) y_{n+1}+\beta_0 y_n-\mu( y_{n+2-m}+(\beta_0-1) y_{n+1-m}+\beta_0 y_{n-m})),
    \label{b0_family}
\end{equation}
with the constraint $0<\beta_0< 1$.
\begin{theorem}
   If the exact solution is sufficiently smooth, the local truncation error of the previous methods has the form
    \begin{equation}
        (T_h)_n^{\beta_0-\text{family}}=-\dfrac{3+\beta_0}{6}y^{(3)}(t_n)h^2+O(h^3).
        \label{order_consistency}
    \end{equation}
    Hence, this family has second order of consistency.
\end{theorem}

The next Definition and Proposition let us justify the assumptions previously made:
\begin{definition}
    For any $z<0$, the curve described by $\Gamma_z$ in \eqref{Gamma_z} delimit a multiple connected region $\Omega$ with multiple components, that is $\mathbb{C}=\Omega \sqcup \Omega^C$ and $\Omega=\Gamma_z \sqcup \Omega_1 \sqcup \dots \sqcup \Omega_r$. If $\mu_j \in \Omega_j$, then $\Omega_j$ is the set of points $\bar{\mu}_j$ such that there exists a continuous curve $\gamma$ connecting $\mu_j$ and $\bar{\mu}_j$, verifying $\gamma \cap \Gamma_z=\emptyset$, thus $\Omega_i \cap \Omega_j = \emptyset$ if $i\not=j$, and $\Omega_i \cap \Gamma_z=\emptyset$ $\forall i$.
    \newline
    We will denote the component containing the value $\mu=0$ by $\tilde{D}_z$, and we will call it the innermost region. 
\end{definition}
\begin{remark}
    For the family of methods considered, as the A-stability condition is imposed, then for any $z<0$ and $\mu=0$, the family is stable (regarding the roots of the polynomial \eqref{char_equation}). 
\end{remark}

\begin{lemma}
\label{lemma:denominator}
Let \(0<\beta_0<1\). Then
\[
 \sigma(e^{\mathrm i\theta})
 =\beta_0+(\beta_0-1)e^{\mathrm i\theta}
   +e^{2\mathrm i\theta}\neq0
 \qquad (0\le\theta\le2\pi).
\]
Moreover, \(\sigma\) is a Schur polynomial.
\end{lemma}

\begin{proof}
Multiplication by \(e^{-\mathrm i\theta}\) gives
\[
 e^{-\mathrm i\theta}\sigma(e^{\mathrm i\theta})
 =
 (1+\beta_0)\cos\theta+\beta_0-1
 +\mathrm i(1-\beta_0)\sin\theta.
\]
If this expression vanished, its imaginary part would imply
\(\sin\theta=0\). For \(\theta=0\) and \(\theta=\pi\), its values are,
respectively, \(2\beta_0\) and \(-2\), so neither is zero.

The Jury conditions for
\(\zeta^2+(\beta_0-1)\zeta+\beta_0\) are
\[
 |\beta_0|<1,\qquad
 1+(\beta_0-1)+\beta_0=2\beta_0>0,\qquad
 1-(\beta_0-1)+\beta_0=2>0.
\]
Thus both roots of \(\sigma\) lie strictly inside the unit disk.
\end{proof}

\begin{proposition}
    Let $\tilde{D}_z$ be the innermost region of $\Gamma_z$, then for the $\beta_0$-family is verified that $\tilde{D}_z=D_z$.
\end{proposition}
\begin{proof}
    We will first show that $\Gamma_z$ is a closed curve; thus, $\bar{D}_z$ is bounded. 
    \newline
    Let us take the expression for $\mu$: 
    \begin{equation}
        \mu_{m,z,\theta}=\big(e^{i\theta}\big)^m \left(1+\big( 1-e^{i\theta} \big)\frac{\big(-1+2\beta_0+e^{i\theta}\big)}{\big( \beta_0+e^{i\theta}(\beta_0-1)+e^{2i\theta} \big)z}\right),\quad \theta\in[0,2\pi],
    \end{equation}
    The domain $[0,2\pi]$ is compact in $\mathbb{R}$, as it is closed and bounded (Heine-Borel theorem). One can directly check that $\mu_{m,z,0}=\mu_{m,z,2\pi}=1$.   And, since the denominator never reaches the value 0 (Lemma \ref{lemma:denominator}), then this function is continuous on the whole domain.
    \newline
    Hence, from a well-known result in analysis, it follows that the image of a continuous function over a compact set is itself compact. We can conclude that $\Gamma_z$ is a compact set in $\mathbb{C}$, thus closed and bounded.
    \newline
    To prove the identity $\bar{D}_z=D_z$, we follow the same arguments as in \cite{articulo_ana}. In the case $\mu=0$, equation \eqref{DDE_general} becomes equation \eqref{A-stability_problem}. Since our family is A-stable, its methods are stable for any $\lambda<0$, that is, $0\in D_z$ $\forall z \in \mathbb{R}_{<0}\cup \{-\infty\}$. Now, using the fact that the polynomial $P(\mu,z,\zeta)$ is continuous and $\Gamma_z$is the curve determined by points with $|\zeta|=1$, we can assure that $0\in \bar{D}_z\subseteq D_z$. The reverse inclusion is deduced via Cauchy's integral theorem, from which it is possible to show that $P(\mu,z,\zeta)$ has all the roots with $|\zeta|<1$ in $\bar{D}_z$ .
\end{proof}
In order to prove our main goal, which is to demonstrate the unconditional stability of the $\beta_0$-family \eqref{b0_family}, let us introduce the concept of winding number and one major result:
\begin{definition}
    Let $\gamma:[a,b]\rightarrow \mathbb{C}$ be a closed curve and $z\notin \Gamma :=\gamma([a,b])$, we define the winding number of $z$ with respect to $\gamma$ as
    \begin{equation}
        W(\gamma,z)\equiv W_{\gamma}(z):=\dfrac{1}{2\pi i}\int_\gamma \dfrac{1}{\xi-z}d\xi.
    \end{equation}
    \end{definition}
    \begin{proposition}
        With the previous notation, the following statements are satisfied:
        \begin{enumerate}
            \item If $\omega \in \mathbb{C}\setminus \Gamma$, then $W_\gamma (\omega)\in \mathbb{Z}$.
            \item If $\omega_1,\omega_2$ are in the same connected component, then $W_\gamma(\omega_1)=W_\gamma(\omega_2)$.
        \end{enumerate}
    \label{prop_connected_comps}\end{proposition}
    We are now able to prove the next result:
    \begin{theorem}
        The region of unconditional stability of the $\beta_0$-family \eqref{b0_family}, with $0<\beta_0<1$, is
        \begin{equation}
            D=D(0,1).
        \end{equation}
        \label{Theorem_unconditional_stab}
    \end{theorem}
\begin{proof}
    We already know the identity $D_{-\infty}=D(0,1)$, as it was argued above and it constitutes the basis of the method's design. Hence we need to see the inclusion $D_{-\infty}\subseteq D_z$ $\forall z<0$.
    \newline
    Consider the following polynomial
    \begin{equation}
        P(\zeta,\bar{z}):=\zeta^m \left(\dfrac{\bar{z}}{1-\bar{z}}\rho(\zeta)+\sigma(\zeta)\right)-\mu \sigma^*(\zeta)
    \end{equation}

and fix a $z\in \mathbb{R}_{<0}\cup\{-\infty\}$, then the corresponding stability region is determined by
\begin{equation}
    D_z=\left\{\mu \in \mathbb{C}:P\left(\zeta,\dfrac{1}{1-z}\right)=0 \Rightarrow|\zeta|<1\right\}.
\end{equation}
Using the argument principle, we obtain the following equivalences:
\begin{equation*}
        \mu \in D_z \iff \scalebox{1.05}{$P\left( \zeta,\frac{1}{1-z} \right)$} \text{ has $m+2$ roots in $D(0,1)$ }\iff \left\{\begin{array}{lcc}
            P\left( \zeta,\frac{1}{1-z} \right) \neq 0 \;\; \forall|\zeta|=1, \\
             m+2=F\left( \frac{1}{1-z}  \right), 
        \end{array}
        \right.
    \end{equation*}
where
\begin{equation}
    F(\bar{z}):=\dfrac{1}{2\pi i}\int_{|\zeta|=1} \dfrac{P_\zeta(\zeta,\bar{z})}{P(\zeta,\bar{z})}d\zeta, \quad P_\zeta(\zeta,\bar{z}) := \dfrac{\partial P(\zeta,\bar{z})}{\partial \zeta}.
\end{equation}
It is possible to verify the second condition $F\big(\frac{1}{1-z}\big)=m+2$ $\forall m\in \mathbb{N}$, $\forall z<0$, for $\mu=0$, and therefore by applying \textbf{Proposition} \ref{prop_connected_comps} this can be extended $\forall \mu \in D_z$. 
\newline
To demonstrate that the first condition $P\left( \zeta,\frac{1}{1-z} \right) \neq 0 \;\; \forall|\zeta|=1$ is satisfied, let us suppose by contradiction that $\exists \mu \in D_{-\infty}$ such that $\mu \notin D_{z_0}$ for some $z_0 \in \mathbb{R}_{<0}$. Observe that $\frac{1}{1-z}<1$ $\forall z<0$ and $P(\zeta,\bar{z})$ is continuous $\forall \bar{z}<1$, in consequence there will be a $\bar{z}_1 \in \left(0,\frac{1}{1-z_0}\right]$ such that $P(\zeta,\bar{z}_1)=0$ for some $|\zeta|=1$. Equivalently, $\mu \in \Gamma_{z_1}$ for some $z_1\in \mathbb{R}_{<0}$ with $z_1<z_0$.
\newline
Furthermore, as $\mu$ belongs to $\Gamma_{z_1}$, its module can be written as
\begin{equation}
        |\mu|=\left|1+\frac{\big( 1-e^{i\theta} \big)\big(-1+2\beta_0+e^{i\theta}\big)}{\big( \beta_0+e^{i\theta}(\beta_0-1)+e^{2i\theta} \big)z_1}  \right|,\; -\pi \le \theta \le \pi,\; 0<\beta_0\le1.
        \label{mod_mu}
    \end{equation}
    To conclude with a contradiction, we just need to check that $|\mu|\ge 1$ (thus $\mu \notin D_{-\infty}$).
    \newline
    By using the following property 
    \begin{equation*}
        |\omega_1+\omega_2|^2=|\omega_1|^2+|\omega_2|^2+2\Re{(\omega_1 \bar{\omega_2})},\quad \text{for any }\omega_1,\omega_2\in \mathbb{C},
    \end{equation*}
    where $\bar{\omega}$ denotes the conjugate of $\omega$, our expression \eqref{mod_mu} becomes
    \begin{equation*}
        |\mu|^2=1+\left| \frac{\big( 1-e^{i\theta} \big)\big(-1+2\beta_0+e^{i\theta}\big)}{\big( \beta_0+e^{i\theta}(\beta_0-1)+e^{2i\theta} \big)z_1} \right|^2+2\Re{\left(\frac{\big( 1-e^{-i\theta} \big)\big(-1+2\beta_0+e^{-i\theta}\big)}{\big( \beta_0+e^{-i\theta}(\beta_0-1)+e^{-2i\theta} \big)z_1} \right)}.
    \end{equation*}
    Hence, we need the third term to be positive, and as $z_1<0$, we end up with
    \begin{equation*}
         \Re{\left( \frac{\big( 1-e^{-i\theta} \big)\big(-1+2\beta_0+e^{-i\theta}\big)}{ \beta_0+e^{-i\theta}(\beta_0-1)+e^{-2i\theta}}\right)} =\left( \frac{4(\beta_0-1)\sin^4(\frac{\theta}{2})}{ 1+\beta_0(\beta_0-1)+(\beta_0^2-1)\cos(\theta)+\beta_0\cos(2\theta)}\right)\le 0,
    \end{equation*}
    for $-\pi \le \theta \le \pi,\; 0<\beta_0<1$. The numerator is trivially non-negative, so we will check that the denominator remains always positive. A simple analysis of its critical points lets us conclude with this.
    \newline
    Finally, we obtain $|\mu|^2\ge1$, reaching the aforementioned contradiction.
\end{proof}
Regarding the application of the method to the general problem \eqref{prob_matricial_general}, and as a consequence of the previous theorem, we have the following result:
\begin{corollary}   
    Assume there exists a $p\in\mathbb{R}$ such that $F_p \subseteq D(0,1)$. Then every method derived from the family \eqref{b0_family} applied to the problem \eqref{prob_matricial_general} is unconditionally stable.
    \label{Theorem_FOV_uncond_stability}
\end{corollary}
\begin{proof}
    A direct consequence of \textbf{Theorem \ref{Theorem_unconditional_stab}} is
    \begin{equation*}
        D(0,1)\subseteq\bigcap\limits_{z\in \mathbb{R}_{<0}\cup\{-\infty\}}D_z\subseteq\bigcap\limits_{z\in F(hA)}D_z.
    \end{equation*}
    These inclusions let us conclude with the proof, as $\mu \in F_p\subseteq D(0,1)$ implies that $\mu\in D_z$ $\forall z\in F(hA)$.
\end{proof}
Finally, recalling what we deduced at the end of the last section, we can assure that any method of this family, applied to \eqref{prob_matricial_general}, will be unconditionally stable if the asymptotic stability conditions \eqref{cond_fov_asint_estab} hold true and $A^{-1}B$ is a normal matrix.


\section{Numerical stability for off-grid and variable delays}
\label{s4:non-constant-delay}

In this section, we will distinguish a constant off-grid delay from a genuinely variable delay. The former provides one autonomous recurrence with a constant $u \neq 0$, whereas the latter gives a sequence of step matrices whose products determine stability.

\subsection{  Numerical stability when the delay is constant but $u \neq 0$  }

First, we will deal with the case $\tau=(m+u)h$, with $m\ge1$ and $u\in [0,1)$.  Thus,
$t_{n+j}-\tau=t_{n-m+j}-uh$. Bellen and Zennaro studied the smooth delayed problem, $y'(t)=f\bigl(t,y(t),y(t-\tau(t))\bigr)$, where $f$ is Lipschitz continuous in its state arguments, $\tau(t)\ge \tau_*>0$, and the mesh contains the derivative-discontinuity points relevant to order $p$.  Let the underlying stable ODE formula have order $p$, and let $\eta_h$ be the continuous reconstruction used to evaluate delayed terms.

 If, when a step is started with exact data and exact past history, this reconstruction satisfies the uniform local estimate
\[
 \max_{t_n\le t\le t_{n+1}}
 || y(t)-\widehat\eta_{n+1}(t)||=O(h_n^{r}),
\]
then, for sufficiently small $h=\max_n h_n$ (in particular, with no within-step overlap), the resulting DDE method satisfies
\[
 \max_n  ||y(t_n)-y_n||
 +\max_{t_0\le t\le T}  ||y(t)-\eta_h(t)||
 =O\!\left(h^{\min\{p,r\}}\right).
\]
Thus a uniform reconstruction error $O(h^p)$ is a standard mesh-independent sufficient condition for preserving the global order $p$ of the discrete formula, see \cite[Th.~4.1.2]{bellen-zennaro-DDE}.

In what follows, we will assume that there are no discontinuities in the DDE. Define the three quadratic Lagrange basis functions
$U_s:[0,1)\to\mathbb{R}$, $s=0,1,2$, by
\begin{equation*}
 U_0(u)=\frac{(u-1)(u-2)}{2},\qquad
 U_1(u)=-u(u-2),\qquad
 U_2(u)=\frac{u(u-1)}{2}.
\end{equation*}
The quadratic backward interpolant
\begin{equation*}
 \mathcal I_u y_q:=U_0(u)y_q+U_1(u)y_{q-1}+U_2(u)y_{q-2}
\end{equation*}
approximates $y(t_q-uh)$. The delayed terms of the method
\eqref{IMEX_general} are therefore approximated by
\begin{equation}
    \sum\limits_{j=0}^k \alpha_j y_{n+j}=h \sum\limits_{j=0}^k \beta_j f(t_{n+j},y_{n+j})+h \sum\limits_{j=0}^{k}\beta^*_j g(t_{n-m+j},\frac{1}{2}(u-1)(u-2)y_{n-m+j}-u(u-2)y_{n-m-1+j}+\frac{1}{2}u(u-1)y_{n-m-2+j}).
    \label{IMEX_interpolado}
\end{equation}
For sufficiently smooth solutions, the interpolation error is
\(O(h^3)\), so it does not reduce the second order of the formula.
\newline
By applying \eqref{IMEX_interpolado} to the problem \eqref{prob_matricial_general}, the following result is obtained
\begin{equation}
    \sum\limits_{j=0}^k \alpha_j y_{n+j}=h\sum\limits_{j=0}^k \beta_j A y_{n+j}+h\sum\limits_{j=0}^{k}\beta_j^*B (\frac{1}{2}(u-1)(u-2)y_{n-m+j}-u(u-2)y_{n-m-1+j}+\frac{1}{2}u(u-1)y_{n-m-2+j}).
    \label{IMEX.interpolado.matricial}
\end{equation}
For $m\ge1$, set
\begin{equation*}
    Y_n:=(y_{n+k},y_{n+k-1},\dots,y_{n-m-1})^T
    \in \mathbb{C}^{(k+m+2)d}.
\end{equation*}
Then $Y_n=W(u)Y_{n-1}$, where
\begin{equation*}
 W(u)=\begin{pmatrix}
 M^{-1}C_1(u)&M^{-1}C_2(u)&\cdots&M^{-1}C_{k+m+2}(u)\\
 I&0&\cdots&0\\
 &\ddots&\ddots&\vdots\\
 0&&I&0
 \end{pmatrix},
\end{equation*}
with $M=\alpha_kI-h\beta_kA$ and
\begin{equation*}
 C_r(u)=h\beta_{k-r}A-\alpha_{k-r}I
 +h\sum_{\ell=0}^2U_\ell(u)\beta^*_{k+m+\ell-r}B,
 \qquad r=1,\dots,k+m+2,
\end{equation*}
using again zero coefficients outside $\{0,\dots,k\}$. This is the
augmented matrix associated with the chosen convention. 
\newline
Analogously to the procedure in \S \ref{sec:2}, we arrive at the next result: for each eigenvalue $\zeta$ of W there will exist an associated eigenvector V with the form
\begin{equation*}
    V=\big(\zeta^{k+m+2}v,\zeta^{k+m+1}v,\dots,\zeta v,v\big)^T \in \mathbb{C}^{(k+m+3)d},
\end{equation*}
with $v\in\mathbb{C}^d$ verifying $T_u(\zeta)v=0$ and
\begin{equation}
    T_u(x)=x^{m+2}\left( \frac{1}{h}\rho(x)I-\sigma(x)A \right)-\sigma^*(x)p_u(x)B,
    \qquad p_u(x):=U_0(u)x^2+U_1(u)x+U_2(u).
    \label{Ec_T(x)_interp}
\end{equation}

\subsection{Stability of the $\beta_0$-familiy for $u\not=0$}
Let us now focus on the previously designed family \eqref{b0_family} with the aforementioned interpolation.
For $u\in[0,1)$, the quadratic Lagrange interpolant defined above is
\begin{equation}
\begin{split}
  \mathcal I_u y_j
  ={}& \frac{(u-1)(u-2)}{2}\,y_j
       -u(u-2)\,y_{j-1}
       +\frac{u(u-1)}{2}\,y_{j-2} .
\end{split}
\label{eq:quadratic-interpolation}
\end{equation}
It approximates $y(t_{j}-uh)$ with an error of order $h^3$, and
therefore does not reduce the order of an order-two method.  For a
modal sequence $y_j=\xi^j$, its symbol is
\begin{equation}
  \mathcal I_u y_j=\xi^{j-2}p_u(\xi),\qquad
  p_u(\xi):=\frac{(u-1)(u-2)}{2}\xi^2-u(u-2)\xi+\frac{u(u-1)}{2}.
  \label{eq:interpolation-symbol}
\end{equation}

Hence, the characteristic equation can be written as
\begin{equation}
  \xi^{m+2}\bigl(\rho(\xi)-z\sigma(\xi)\bigr)
  +z\mu\,\sigma^*(\xi)p_u(\xi )=0.
  \label{eq:characteristic-equation}
\end{equation}

 Equation \eqref{eq:characteristic-equation} is the scalar reduction of
\eqref{Ec_T(x)_interp}; the interpolation factor $p_u$ is the only new
factor in comparison with \eqref{char_equation_FOV}.
For fixed $z\ne0$, define the boundary locus
\begin{equation}
  \Pi_{z,u}
  :=\left\{
  \mu_u(\xi)
  =-\frac{\xi^{m+2}\bigl(\rho(\xi)-z\sigma(\xi)\bigr)}
  {z\sigma^*(\xi)p_u(\xi )}
  : |\xi|=1
  \right\},
  \label{eq:boundary-locus}
\end{equation}
where values for which the denominator vanishes are interpreted as
points at infinity, and let
\begin{equation}
  \sigma_{z,u}:=\inf\bigl\{|\mu|:\mu\in\Pi_{z,u}\bigr\}.
  \label{eq:boundary-radius}
\end{equation}

\begin{lemma}
\label{thm:quadratic-interpolation}
For every $u\in[0,1)$ and every $\xi\in\C$ with $|\xi|=1$,
\begin{equation}
  |p_u(\xi)|\le 1.
  \label{eq:symbol-bound}
\end{equation}
More precisely, if $\xi=e^{\mathrm i\theta}$, then
\begin{equation}
  |p_u(e^{\mathrm i\theta})|^2
  =1-u(2-u)(1-u)^2(1-\cos\theta)^2.
  \label{eq:exact-modulus}
\end{equation}
Consequently, for every fixed $z$ for which the boundary loci are
defined,
\begin{equation}
  |\mu_u(\xi)|\ge |\mu_0(\xi)|
  \qquad (|\xi|=1),
  \label{eq:pointwise-comparison}
\end{equation}
and hence
\begin{equation}
  \sigma_{z,u}\ge \sigma_{z,0}.
  \label{eq:radius-comparison}
\end{equation}
\end{lemma}

\begin{proof}
Write
\[
  a_u=\frac{(u-1)(u-2)}{2},\qquad
  b_u=-u(u-2),\qquad
  c_u=\frac{u(u-1)}{2}.
\]
Since these coefficients are real, for $w=e^{\mathrm i\theta}$ we have
\begin{align*}
  |p_u(e^{\mathrm i\theta})|^2
  & =a_u^2+b_u^2+c_u^2
      +2b_u(a_u+c_u)\cos\theta
      +2a_uc_u\cos(2\theta) \\
  & =1-u(2-u)(1-u)^2(1-\cos\theta)^2.
\end{align*}
Because $0\le u<1$, all the factors in
\[
  1-|p_u(e^{\mathrm i\theta})|^2
  =u(2-u)(1-u)^2(1-\cos\theta)^2
\]
are nonnegative.  This proves \eqref{eq:symbol-bound}.

Now let $|\xi|=1$.  Since $|\xi^{m+2}|=1$, equation
\eqref{eq:boundary-locus} gives
\begin{equation}
  |\mu_u(\xi)|
  =\frac{|\rho(\xi)-z\sigma(\xi)|}
  {|z\sigma^*(\xi)|\,|p_u(\xi )|}.
  \label{eq:mu-modulus}
\end{equation}
For $u=0$, $p_0(\xi)=\xi^2$, and therefore
$|p_0(\xi )|=1$.  The estimate \eqref{eq:symbol-bound} thus yields
\[
  |\mu_u(\xi)|
  \ge
  \frac{|\rho(\xi)-z\sigma(\xi)|}{|z\sigma^*(\xi)|}
  =|\mu_0(\xi)|.
\]
At a zero of $p_u(\xi )$, the left-hand side is understood as
$+\infty$, so the same comparison remains valid.  Taking the infimum
over the unit circle proves \eqref{eq:radius-comparison}.
\end{proof}

\begin{theorem}[Fixed off-grid stability]
\label{thm:fixed-off-grid-stability}
Let \(0<\beta_0<1\), \(m\ge 1\), \(u\in[0,1)\), and
\(z\in\mathbb R_{<0}\cup\{-\infty\}\). If \(|\mu|<1\), then every
root of \eqref{eq:characteristic-equation}, interpreted through its
normalized limit for \(z=-\infty\), belongs to \(D(0,1)\).
Consequently, quadratic interpolation does not reduce the stability
disk \(D(0,1)\) for a fixed non-grid-aligned delay.
\end{theorem}


\begin{proof}
For finite \(z<0\), let
\[
 R_z(\xi)=\rho(\xi)-z\sigma(\xi).
\]
Since the underlying multistep formula is $A$-stable, its boundary
locus satisfies
$\operatorname{Re}(\rho(\xi)\overline{\sigma(\xi)})\ge0$ for
$|\xi|=1$. Hence
\[
\begin{aligned}
 |R_z(\xi)|^2-|z\sigma(\xi)|^2
 &=|\rho(\xi)|^2
 -2z\operatorname{Re}\!\left(
 \rho(\xi)\overline{\sigma(\xi)}\right)\ge0.
\end{aligned}
\]
Therefore, using \(|p_u(\xi)|\le1\) and \(|\mu|<1\),
\[
 |z\mu\sigma(\xi)p_u(\xi)|
 <|\xi^{m+2} R_z(\xi)|,\qquad |\xi|=1.
\]
Rouch\'e's theorem shows that the characteristic polynomial has the
same number of roots in \(D(0,1)\) as \(\xi^{m+2} R_z(\xi)\). The latter
has \(m+2\) roots at zero and the two stable roots of the underlying
\(A\)-stable formula, so all \(m+4\) roots are interior.

For \(z=-\infty\), the normalized polynomial is
\[
 \sigma(\xi)\bigl(\xi^{m+2}-\mu p_u(\xi)\bigr).
\]
The polynomial \(\sigma\) is Schur, and Rouch\'e's theorem applied to
\(\xi^{m+2}-\mu p_u(\xi)\) completes the proof.
\end{proof}

\subsection{  Variable delays  }

At every grid point write, with the same convention as above,
\begin{equation*}
 \tau_\ell=(m_\ell+u_\ell)h,\qquad
 m_\ell=\left\lfloor\frac{\tau_\ell}{h}\right\rfloor\ge1,
 \qquad 0\le u_\ell<1,
\end{equation*}
and define
\begin{equation*}
 \mathcal I_\ell y:=\sum_{r=0}^2U_r(u_\ell)y_{\ell-m_\ell-r}.
\end{equation*}
Thus, in a two-step formula the three delayed values are reconstructed
with $\mathcal I_n y$, $\mathcal I_{n+1}y$, and
$\mathcal I_{n+2}y$, respectively; in general, the delayed term at
stage $j$ in \eqref{IMEX_interpolado} is
$g(t_{n+j},y_{n+j},\mathcal I_{n+j}y)$.

If \(A=A^*<0\) and
\(F_p(A,B)\subset D(0,1)\) for some \(p\in\mathbb R\), the
field-of-values reduction and
Theorem~\ref{thm:fixed-off-grid-stability} prove stability for every
fixed admissible pair \((m,u)\).  However, for a genuinely variable delay, the three delayed evaluations in a
two-step formula generally use different pairs
\[
 (m_n,u_n),\qquad(m_{n+1},u_{n+1}),\qquad(m_{n+2},u_{n+2}).
\]
Assume \(0<\tau_{\min}\le\tau_\ell\le\tau_{\max}\), fix \(h\), and
embed all required history values in a fixed-dimensional vector
\(X_n\). The homogeneous recurrence can then be written as
\begin{equation}
 X_{n+1}=W(\vartheta_n;h)X_n,
 \label{eq:variable-augmented-recurrence}
\end{equation}
where \(\vartheta_n\) contains now the three stage-dependent delay
parameters and any other varying coefficients. Consequently, stability
is governed by products of the matrices \(W(\vartheta_n;h)\), not only by
the spectra of the individual frozen matrices.

\begin{theorem}[Common quadratic Lyapunov certificate]
\label{thm:common-lyapunov}
Let \(\mathcal W_h\) denote the family of all admissible step matrices.
If there exist \(P_h=P_h^*\succ0\) and \(q_h\in(0,1)\) such that
\[
 W^*P_hW\preceq q_h^2P_h
 \qquad\text{for every }W\in\mathcal W_h,
\]
then every admissible variable-delay sequence is uniformly
exponentially stable and
\[
 \|X_n\|_{P_h}
 \le q_h^{\,n-n_0}\|X_{n_0}\|_{P_h}.
\]
\end{theorem}

\begin{proof}
At each step,
\[
 \|X_{n+1}\|_{P_h}^2
 =X_n^*W(\vartheta_n;h)^*P_hW(\vartheta_n;h)X_n
 \le q_h^2\|X_n\|_{P_h}^2.
\]
Iteration gives the stated estimate.
\end{proof}

Unfortunately, stability of every ``frozen recurrence'' alone does not imply stability of the products.  As we can check in the next examples, in practice, in many numerical problems unconditional stability of the method for any frozen $u \in [0,1)$ provides stability for large time steps, however, a  Lyapunov certificate (or a similar technique) should be verified for every step size  to demonstrate stability without any doubts.

\subsection{Example: linear system where matrices simultaneously diagonalize}
    Consider the following problem:
    \begin{equation}
    \left\{
    \begin{array}{cc}
         y'(t)=Ay(t)+s(t)By(t-g(t))+\tilde{f}(t),& t\ge 0,  \\
         y(t)=\left(\begin{gathered}
    \cos\left(\frac{t}{25}\right) \\
    \sin^2\left(\frac{t}{10}\right)-e^{-\frac{t}{50}} \\
    2+\frac{t}{100}
  \end{gathered}
              \right), & t\le 0, 
    \end{array}
    \right.
    \label{test1}
\end{equation}
\textit{where the matrices and the functions s(t) and g(t) are defined as}
\begin{equation*}
    A=\begin{pmatrix}
        -200&0&0  \\
         340&-30&0  \\
         150&-20&-10
    \end{pmatrix},
    \quad
    B=\begin{pmatrix}
        -158&0&0 \\
        276&-20&0 \\
        127.99&-10.01&-9.99
    \end{pmatrix}, \quad s(t)=1-e^{-t}, \quad g(t)=3/4+\cos(t)/2,
\end{equation*}
\textit{and  $\tilde{f}(t)$ is such that } $y(t)=\left(\cos(\frac{t}{25}),\sin^2(\frac{t}{10})-e^{-\frac{t}{50}},2+\frac{t}{100}\right)^T$
\textit{ is the solution to the problem.}

\textbf{Solution.} $\tilde{f}(t)$ is determined by $\tilde{f}(t)=y'(t)-Ay(t)-s(t)By(t-g(t))$:
\begin{equation}
\begin{aligned}
\tilde{f}(t)=\Bigg(
&200 \cos\left(\frac{t}{25}\right) - (1 - e^{-t}) \left(0. - 158. \cos\left(\frac{1}{25} \left(-\frac{3}{4} + t - \frac{\cos(t)}{2}\right)\right)\right) - \frac{1}{25} \sin\left(\frac{t}{25}\right), \\[10pt]
&\frac{e^{-\frac{t}{50}}}{50} - 340 \cos\left(\frac{t}{25}\right) + \frac{1}{5} \cos\left(\frac{t}{10}\right) \sin\left(\frac{t}{10}\right) + 30 \left(-e^{-\frac{t}{50}} + \sin^2\left(\frac{t}{10}\right)\right) \\
&\quad - (1 - e^{-t}) \left(0. + 276. \cos\left(\frac{1}{25} \left(-\frac{3}{4} + t - \frac{\cos(t)}{2}\right)\right) - 20. \left(-e^{\frac{1}{50} \left(\frac{3}{4} - t + \frac{\cos(t)}{2}\right)} + \sin^2\left(\frac{1}{10} \left(-\frac{3}{4} + t - \frac{\cos(t)}{2}\right)\right)\right)\right), \\[10pt]
&\frac{1}{100} + 10 \left(2 + \frac{t}{100}\right) - 150 \cos\left(\frac{t}{25}\right) + 20 \left(-e^{-\frac{t}{50}} + \sin^2\left(\frac{t}{10}\right)\right) \\
&\quad - (1 - e^{-t}) \bigg(-9.99 \left(2 + \frac{1}{100} \left(-\frac{3}{4} + t - \frac{\cos(t)}{2}\right)\right) + 127.99 \cos\left(\frac{1}{25} \left(-\frac{3}{4} + t - \frac{\cos(t)}{2}\right)\right) \\
&\quad - 10.01 \left(-e^{\frac{1}{50} \left(\frac{3}{4} - t + \frac{\cos(t)}{2}\right)} + \sin^2\left(\frac{1}{10} \left(-\frac{3}{4} + t - \frac{\cos(t)}{2}\right)\right)\right)\bigg)\Bigg)^T.
\end{aligned}
\end{equation}

By applying the interpolated method \eqref{IMEX.interpolado.matricial} to this problem and choosing $\beta_0=1/2$, we end up with:
\begin{equation}
\begin{aligned}
    y_{n+2}-y_{n+1}={}&h\bigg( A(y_{n+2}-\frac{1}{2}y_{n+1}+\frac{1}{2}y_{n})
    +\frac{1}{2}s(t_n)B\mathcal I_n y
    -\frac{1}{2}s(t_{n+1})B\mathcal I_{n+1}y  \\
    +  & s(t_{n+2})B\mathcal I_{n+2}y   +\frac{1}{2}\widetilde f(t_n)  -\frac{1}{2}\widetilde f(t_{n+1})
    +\widetilde f(t_{n+2})\bigg).
    \end{aligned}
\end{equation}

To ensure $m_\ell\ge1$ at every stage, it is enough to impose
$h\le\tau_{\min}$. Since the minimum value of the delay is
$\tau_{\min}=0.25$, the no-overlap condition is $h\le0.25$.

In Table \ref{Table:error_test1}, the errors related to each component of the solution in the last iteration are presented.

\begin{table}[h]
\centering
\renewcommand{\arraystretch}{1.3} 
\begin{tabularx}{\textwidth}{|c|*{3}{>{\centering\arraybackslash}X}|}
\hline
\textit{h} & \multicolumn{3}{c|}{\textit{\textbf{Interpolated IMEX}}} \\
\hline
0.5  & $1.793 \cdot 10^{-5}$ & $1.126 \cdot 10^{-4}$ & $3.071 \cdot 10^{-4}$ \\
\hline
0.2  & $4.111 \cdot 10^{-9}$ & $1.740 \cdot 10^{-6}$ & $2.689 \cdot 10^{-6}$ \\
\hline
0.08  & $2.338 \cdot 10^{-10}$ & $2.754 \cdot 10^{-7}$ & $4.153 \cdot 10^{-7}$ \\
\hline
0.05 & $8.854 \cdot 10^{-11}$ & $1.058 \cdot 10^{-7}$ & $1.581 \cdot 10^{-7}$ \\
\hline
0.02 & $4.958 \cdot 10^{-12}$ & $1.736 \cdot 10^{-8}$ & $2.725 \cdot 10^{-8}$ \\
\hline
\makecell{Conv.\\Ratio} & \multicolumn{3}{c|}{1.97351} \\
\hline
\end{tabularx}

\caption{Example test: Errors committed in the last iteration for every component of the solution and the convergence ratio}
\label{Table:error_test1}
\end{table}

The convergence ratio has been defined as:
\begin{equation}
    \log_{\frac{h_1}{h_2}}\left(\dfrac{err_{h_1}}{err_{h_2}}\right),
    \label{convergence_ratio}
\end{equation}
where $h_1=0.08$ and $h_2=0.02$, both below the abovementioned threshold.

\subsection{Example: linear system where the delay depends on the solution}

    Consider the following problem:
    \begin{equation}
    \left\{
    \begin{array}{cc}
         y'(t)=Ay(t)+By(t-g(t,y))+\tilde{f}(t),& t\ge 0,  \\
         y(t)=\left(\begin{gathered}
    \sin\left(\frac{t}{5}\right) \\
    \frac{\cos\left(\frac{t}{3}\right)}{1+\left( \frac{t}{10}\right)^2} \\
    2+\sin\left(\frac{t}{10}\right)
  \end{gathered}
              \right), & t\le 0, 
    \end{array}
    \right.
    \label{test1var}
\end{equation}
\textit{where the matrices and the functions s(t) and g(t) are defined as}
\begin{equation*}
    A=\begin{pmatrix}
        -4&1&0  \\
         1&-4&1  \\
         0&1&-4
    \end{pmatrix},
    \quad
    B=\begin{pmatrix}
        0&1/2&0 \\
        0&0&1/2 \\
        0&0&0
    \end{pmatrix}, \quad  g(t)=2+\sin(y_3(t)),
\end{equation*}
\textit{being $y_3$ the third component of the solution vector $y$ and  $\tilde{f}(t)$ such that } $y(t)=\left(\sin\left(\frac{t}{5}\right), \frac{\cos\left(\frac{t}{3}\right)}{1+\left( \frac{t}{10}\right)^2},2+\sin\left(\frac{t}{10}\right)\right)^T$
\textit{ is the solution to the problem.}

\textbf{Solution. } Analogously as in the previous example,  $\tilde{f}(t)=y'(t)-Ay(t)-By(t-g(t,y(t)))$:
\begin{equation}
\begin{aligned}
\tilde{f}(t) = \bigg(&\frac{1}{5} \cos\left(\frac{t}{5}\right) - \frac{\cos\left(\frac{t}{3}\right)}{1 + \frac{t^2}{100}} + 4 \sin\left(\frac{t}{5}\right) - \frac{0.5 \cos\left(\frac{1}{3}\left(t - 2 - \sin\left(2 + \sin\left(\frac{t}{10}\right)\right)\right)\right)}{1 + \frac{1}{100}\left(t - 2 - \sin\left(2 + \sin\left(\frac{t}{10}\right)\right)\right)^2}, \\[0.8em]
&-2 - \frac{t \cos\left(\frac{t}{3}\right)}{50 \left(1 + \frac{t^2}{100}\right)^2} + \frac{4 \cos\left(\frac{t}{3}\right)}{1 + \frac{t^2}{100}} - \sin\left(\frac{t}{10}\right) - \sin\left(\frac{t}{5}\right) - \frac{\sin\left(\frac{t}{3}\right)}{3 \left(1 + \frac{t^2}{100}\right)} - 0.5 \left(2 + \sin\left(\frac{1}{10}\left(t - 2 - \sin\left(2 + \sin\left(\frac{t}{10}\right)\right)\right)\right)\right), \\[0.8em]
&\frac{1}{10} \cos\left(\frac{t}{10}\right) - \frac{\cos\left(\frac{t}{3}\right)}{1 + \frac{t^2}{100}} + 4 \left(2 + \sin\left(\frac{t}{10}\right)\right)\bigg)^T
\end{aligned}
\end{equation}

Again, by applying the interpolated method \eqref{IMEX.interpolado.matricial} to this problem and now choosing $\beta_0=1/5$, we end up with:
\begin{equation}
\begin{aligned}
    y_{n+2}-\frac{8}{5}y_{n+1}+\frac{3}{5}y_n={}&h\bigg( A(y_{n+2}-\frac{4}{5}y_{n+1}+\frac{1}{5}y_{n})
    +\frac{1}{5}B\mathcal I_n y
    -\frac{4}{5}B\mathcal I_{n+1}y
    +B\mathcal I_{n+2}y     +\frac{1}{5}\widetilde f(t_n)-\frac{4}{5}\widetilde f(t_{n+1})
    +\widetilde f(t_{n+2})\bigg).
    \end{aligned}
\end{equation}

In this case, as $\tau\ge1$ $\forall t$, it suffices to take the following upperbound for the stepsize $h \le 1$.
In Table \ref{Table:error_test2}, the errors related to each component of the solution in the last iteration are presented.

\begin{table}[h]
\centering
\renewcommand{\arraystretch}{1.3} 
\begin{tabularx}{\textwidth}{|c|*{3}{>{\centering\arraybackslash}X}|}
\hline
\textit{h} & \multicolumn{3}{c|}{\textit{\textbf{Interpolated IMEX}}} \\
\hline
0.5  & $1.419 \cdot 10^{-5}$ & $3.434 \cdot 10^{-5}$ & $6.993 \cdot 10^{-5}$ \\
\hline
0.2  & $3.885 \cdot 10^{-6}$ & $4.132 \cdot 10^{-6}$ & $1.110 \cdot 10^{-5}$ \\
\hline
0.08  & $1.013 \cdot 10^{-6}$ & $5.691 \cdot 10^{-7}$ & $1.769 \cdot 10^{-6}$ \\
\hline
0.05 & $4.339 \cdot 10^{-7}$ & $2.137 \cdot 10^{-7}$ & $6.902 \cdot 10^{-7}$ \\
\hline
0.02 & $7.550 \cdot 10^{-8}$ & $3.280 \cdot 10^{-8}$ & $1.103 \cdot 10^{-7}$ \\
\hline
\makecell{Conv.\\Ratio} & \multicolumn{3}{c|}{1.95698} \\
\hline
\end{tabularx}

\caption{Example test: Errors committed in the last iteration for every component of the solution and the convergence ratio}
\label{Table:error_test2}
\end{table}

The stepsizes used in the computation of the convergence ratio \eqref{convergence_ratio} are $h_1=0.2$ and $h_2=0.02$. 
\newline
The convergence ratio is reasonably similar to the first example and close to 2, verifying the order of consistency \eqref{order_consistency}.

\section{Numerical simulations}
\label{sec:numer-simul}
\subsection{Fisher-Kolmogorov-Petrovsky-Piskunov equation}

Let us consider the FKPP equation with $q=2$:

\begin{equation}
    \dfrac{\partial u}{\partial t}=\lambda \dfrac{\partial^2 u}{\partial x^2}+\mu u(t-\tau,x)(1-u(t-\tau,x)^2), \quad t\ge 0,\quad 0\le x \le 1,
    \label{Ec.Fisher-Kolmogorov-Petrovsky-Piskunov}
\end{equation}

where $\lambda >0, \mu>0$ and with the following boundary and initial conditions

\begin{equation*}
    \left\{
    \begin{array}{lcc}
         u(t,0)=u(t,1)=0,&  \\
         u(t,x)=\sin(\pi x),& t\le 0. 
    \end{array}
    \right.
\end{equation*}

\textbf{Solution.} We will approach this using the so called \textit{Method of lines}, which consists basically of discretizing the space variable x, ending up with a system of DDEs to which we apply the theory presented in the previous section.

First, let us define $\Delta x:=\frac{1}{N}$, and the mesh determined by $x_0=0<\dots<x_j=j\Delta x<\dots<x_N=1$. Denoting by $z^j(t)$ the approximation of $u(t,x_j)$, and the vector $z(t)=\big(z^1(t),\dots,z^{N-1}\big)^T$, then the problem \eqref{Ec.Fisher-Kolmogorov-Petrovsky-Piskunov} is reduced to
\begin{equation}
    z'(t)=Az(t)+g(z(t-\tau)),
    \label{Ec.EDPR_IMEXincond2_con_g}
\end{equation}
with
\begin{equation*}
    A=\dfrac{\lambda}{(\Delta x)^2}\begin{pmatrix}
        -2&1&0&\dots&0\\
        1&-2&1&\dots&0\\
        0&1&-2&\dots&0\\
        \vdots&\vdots&\vdots&\ddots&\vdots\\
        0&0&\dots&1&-2
    \end{pmatrix},
\end{equation*}

and the j-th component of $g(z(t-\tau))$ being

\begin{equation*}
    g_j(z(t-\tau))^T=\mu z^j(t-\tau) \big( 1-z^j(t-\tau)^2 \big).
\end{equation*}

The next step would be to transform equation \eqref{Ec.EDPR_IMEXincond2_con_g} into a matrix one, with the same form as \eqref{prob_matricial_general}, and thus use the results we stated.

We can directly see that, for every $j=1,\dots,N-1$, the component $g_j$ of $g(z(t-\tau))$ is differentiable and vanishes for $z\equiv$. Therefore, by the Mean Value Theorem, we can write $g_j(z(t-\tau))=g_j'(c_j)z^j(t-\tau)$, for some $c_j \in (0,z^j)$.

Consequently, we arrive at
\begin{equation*}
    z'(t)=Az(t)+B(c)z(t-\tau)
,
\end{equation*}
with $c=(c_1,\dots,c_{N-1})$, and 
\begin{equation*}
    B=\mu \begin{pmatrix}
        1-3c_1^2&0&\dots&0
        \\0&1-3c_2^2&\dots&0
        \\ \vdots&\vdots&\ddots&\vdots\\
        0&0&\dots&1-3c_{N-1}^2
    \end{pmatrix}.
\end{equation*}

In this sense, let us take $\beta_0=\frac{1}{2}$, then the matrix equation to deal with would be:
\begin{equation}
    z_{n+2}-z_{n+1}=h\Big(A(z_{n+2}-\frac{1}{2}z_{n+1}+\frac{1}{2}z_n)+B(c)(z_{n-m+2}-\frac{1}{2}z_{n-m+1}+\frac{1}{2}z_{n-m})\Big).
    \label{Ec.Fisher_metodo}
\end{equation}

Generally, A and B defined in this way do not commute, hence we will make use of the theory related to the Field of Values. If $\exists p\in \mathbb{R}$ such that $F_p \subseteq D(0,1)$, then by \textbf{Theorem \ref{Theorem_FOV_uncond_stability}} our method will be unconditionally stable. Thus, we need to study for which parameters the condition is satisfied.

Let us assume that for all n the following holds
\begin{equation}
    |z_{n-m}^j|\le 1,\quad j=1,\dots,N-1.
    \label{Ec.Fisher_cond_z_j}
\end{equation}
Then $-2\mu<g'_j(c_j)<4\mu$, and using the fact that B is Hermitian, we have that
\begin{equation*}
    F(B(c))=\left[\min\limits_j g'_j(c_j),\max\limits_j g'_j(c_j)\right]\subseteq(-2\mu,4\mu).
\end{equation*}

On the other hand, let us describe the so called Toeplitz matrix M, which has the following form

\begin{equation*}
    M=\begin{pmatrix}
        a&b&0&0&\dots&0\\
        c&a&b&0&\dots&0\\
        0&c&a&\ddots&&\vdots \\
        \vdots&&\ddots&\ddots&\ddots&0\\
        \vdots&&&\ddots&a&b\\
        0&\dots&\dots&0&c&a
    \end{pmatrix}_{N\cdot N},
\end{equation*}
and its eigenvalues are the following
\begin{equation*}
    \bar{\lambda_i}=a-2\sqrt{bc}\cos\Big( \frac{i\pi}{N+1}\Big),\quad i=1,\dots,N. 
\end{equation*}

Taking this into account, the eigenvalues of A will be
\begin{equation*}
    \lambda_j\in \sigma(A) \Longrightarrow \lambda_j=-\frac{2\lambda}{(\Delta x)^2}\left(1 - \cos\left(\frac{j\pi} {N}\right)\right).
\end{equation*}

Using the fact that the spatial stepsize is $\Delta x=\frac{1}{N}$, we have the following inclusion
\begin{equation*}
    \sigma(A)\subseteq\left[-2\lambda N^2\left(1 - \cos\left(\frac{\pi}{N}\right)\right),-2\lambda N^2\left(1 - \cos\left(\frac{(N-1)\pi}{N}\right)\right)\right].
\end{equation*}

From the fact that A is Hermitian and negative definite, it follows that $F(A)=Co(\sigma(A))$, and hence $F(A)$ will also be contained in the previous interval.

If condition \eqref{Ec.Fisher_cond_z_j} is verified, it is possible to follow an argument analogous to that of \cite{LMM-Seib} (see Proposition 6.1) to arrive at
\begin{equation*}
\begin{aligned}
    &F((-A)^{-1/2}B(c)(-A)^{1/2})=\{\langle v,B(c)v\rangle:  \langle v,Av\rangle=1,v\in \mathbb{C}^{N-1} \}
    \\&=\{\langle v,B(c)v\rangle:\langle v,v\rangle=\frac{1}{l},l\in F(A),v\in \mathbb{C}^{N-1}\}
    \\&\subseteq \left[-\frac{4\mu}{2\lambda N^2\left(1 - \cos\left(\frac{\pi}{N}\right)\right)},\frac{2\mu}{2\lambda N^2\left(1+\cos\left(\frac{\pi}{N}\right)\right)}\right].
    \end{aligned}
\end{equation*}
Furthermore, the condition $|z^j(t)|\le 1$ holds for $t\le 0$ and $i=1,\dots,N-1$. Therefore, we could prove by induction that if
\begin{equation*}
    \left|-\frac{4\mu}{4\lambda N^2 \left(1 - \cos\left(\frac{\pi}{N}\right)\right)}\right|<1 
\end{equation*}
is satisfied, then \eqref{Ec.Fisher_cond_z_j} is verified. Hence, the method \eqref{Ec.Fisher_metodo} will be unconditionally stable if
\begin{equation}
    \lambda>\frac{\mu}{N^2 \left(1 - \cos\left(\frac{\pi}{N}\right)\right)}
    \label{Cond_param_Ec_KPP}
\end{equation}
holds.

By taking, for example, $\lambda=1/2$, $\mu=3$, and $N=100$, we would ensure that the condition \eqref{Cond_param_Ec_KPP} is satisfied. Additionally, we assume the following values: the delay $\tau=1$ and the initial time $t_0=0$.

In Figs. \ref{fig:sol_u_EDPR_IMEX_incond2_uh01t2}-\ref{fig:sol_u_EDPR_IMEX_incond2_uh05t50},
we can observe the stability of the numerical solution for specific values of the stepsize h and final time $t_f$. 
\begin{figure}[H]
    \centering 
    \begin{subfigure}[b]{0.49\textwidth}
        \centering
        \includegraphics[width=\textwidth]{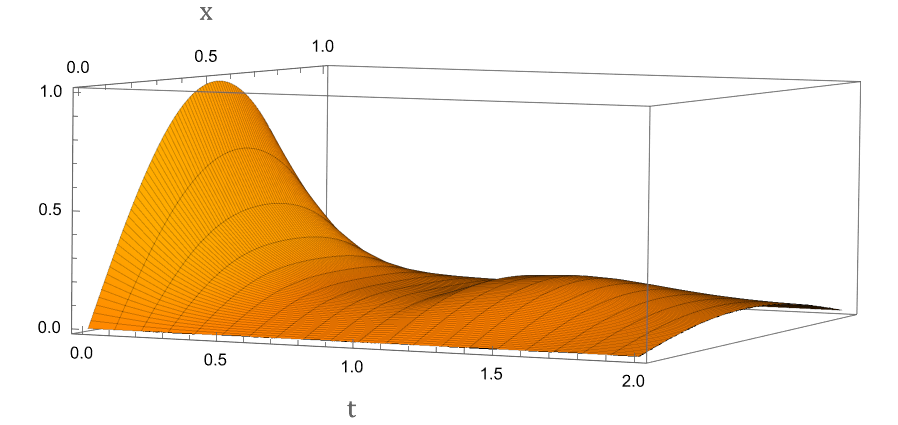}
        \caption{$u$ for $h=0.1,\;t_f=2$}
        \label{fig:sol_u_EDPR_IMEX_incond2_uh01t2}
    \end{subfigure}
    \hfill
    \begin{subfigure}[b]{0.49\textwidth}
        \centering
        \includegraphics[width=\textwidth]{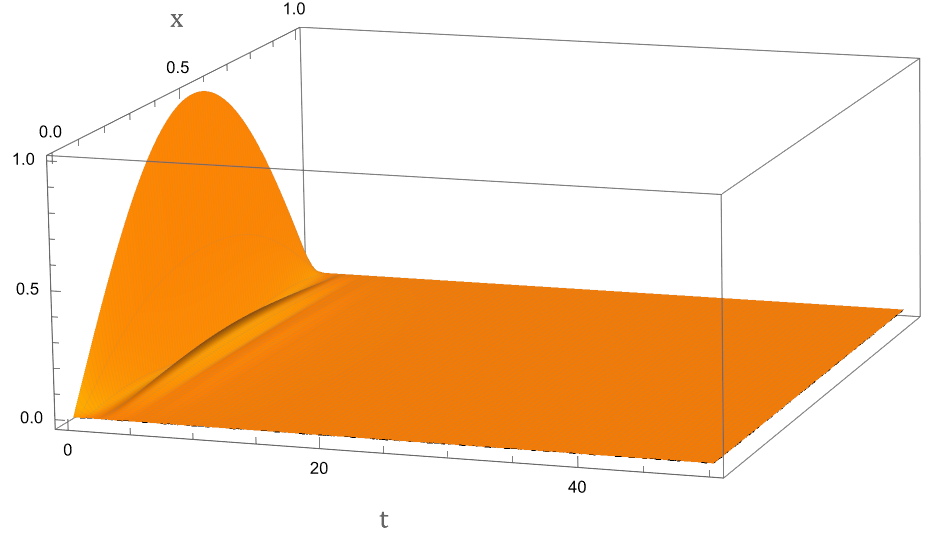}
        \caption{$u$ for $h=0.5,\;t_f=50$}
        \label{fig:sol_u_EDPR_IMEX_incond2_uh05t50}
    \end{subfigure}

    \caption{Numerical solution of the problem  \eqref{Ec.Fisher-Kolmogorov-Petrovsky-Piskunov}, given by the $\beta_0=1/2$ method and particular values of $h$ and $t_f$.}
    \label{fig:sol_u_EDPR_IMEX_incond2} 
\end{figure}

\subsection{Linear system of PDDEs with simultaneous diagonalization criterion}
\label{EDPR_mat_diag_simult}

Consider the following problem:
\begin{equation}
    \left\{
    \begin{array}{cc}
          \dfrac{\partial u_1}{\partial t}=a_1 \dfrac{\partial^2 u_1}{\partial x^2}-e^{l\frac{\pi}{2}}u_1(t-\tau,x)+\frac{1}{4}e^{l\frac{\pi}{2}}(4l+\pi^2)u_2(t-\tau,x) & t\ge 0, \\\\
          \dfrac{\partial u_2}{\partial t}=a_2 \dfrac{\partial^2 u_2}{\partial x^2}-e^{l\frac{\pi}{2}}u_2(t-\tau,x)-\frac{1}{4}e^{l\frac{\pi}{2}}(4l+\pi^2)u_1(t-\tau,x) & t\ge 0,
    \end{array}
    \right.
    \label{Problema_EDPR_lineal}
\end{equation}
for $-1\le x\le 1$, and the following boundary and initial conditions
\begin{equation*}
    \left\{
    \begin{array}{cc}
         u_1(t,-1)=u_1(t,1)=0,  \\
         u_2(t,-1)=u_2(t,1)=0,\\
         u_1(t,x)=\phi_1(t,x), \quad t\le 0,
         \\
         u_2(t,x)=\phi_2(t,x),\quad t\le 0,
    \end{array}
    \right.
\end{equation*}
being $a_1,a_2\in \mathbb{R}^+$, $l\in \mathbb{R}$, and $\phi_1(t,x),\phi_2(t,x)$ two given functions. 

\textbf{Solution. } In the theoretical framework Sec.\ref{sec:2} we assumed DDEs with respect to just the time variable t. Therefore, to utilize the mentioned results we will make use of the socalled Method of Lines.

This method is based on approximating the spatial derivatives in a finite number of spatial nodes. In this way, the space dependence is discretized, giving as a result a system of DDEs where just time remains as a variable.

Let $n\in \mathbb{N}$, $x_0=-1$, and $x_n=1$. Let us define $\Delta x=\dfrac{x_n-x_0}{n}=\dfrac{2}{n}$; thus, we build the mesh $x_0=-1 < \dots < x_j=x_0+j\cdot \Delta x < \dots < x_n=1$. In each of the nodes, we will use the second order centered formula for the second derivative:

\begin{equation*}
    \left.\dfrac{\partial^2 u_k(t,x)}{\partial x^2}\right|_{x=x_j}=\dfrac{u_{i,{j-1}}(t)-2u_j(t)+u_{j+1}(t)}{(\Delta x)^2}+\mathcal{O}((\Delta x)^2), \quad k=1,2.
\end{equation*}

Consequently, we arrive at:
\begin{equation*}
\left\{
\begin{aligned}
     \frac{\partial u_{1,j}(t)}{\partial t}=&a_1 \frac{u_{1,j-1}(t)-2u_{1,j}(t)+u_{1,j+1}(t)}{(\Delta x)^2}-e^{l\frac{\pi}{2}}u_{1,j}(t-\tau,x)+\frac{1}{4}e^{l\frac{\pi}{2}}(4l+\pi^2)u_{2,j}(t-\tau,x),  \\
     \frac{\partial u_{2,j}(t)}{\partial t}=&a_2 \frac{u_{2,j-1}(t)-2u_{2,j}(t)+u_{2,j+1}(t)}{(\Delta x)^2}-e^{l\frac{\pi}{2}}u_{2,j}(t-\tau,x)-\frac{1}{4}e^{l\frac{\pi}{2}}(4l+\pi^2)u_{1,j}(t-\tau,x),
\end{aligned}
\right.    
\end{equation*}
for $t\ge 0$ $-1\le x\le 1$, and the conditions
\begin{equation*}
    \left\{
    \begin{array}{lcc}
         u_{1,-1}(t)=u_{1,1}(t)=0,  \\
         u_{2,-1}(t)=u_{2,1}(t)=0, \\
         u_{1,j}(t)=\phi_1(t,x_j),\quad t\le 0, \\
         u_{2,j}(t)=\phi_2(t,x_j),\quad t\le 0,
    \end{array}
    \right.
\end{equation*}

Let us denote by $z_{k,j}(t)$, with $k=1,2$, the approximation of $u_k(t,x_j)$, then take the vector $z_k(t):=\big(z_{k,1}(t),\dots,z_{k,n-1}(t)\big)^T$, and the system is reduced to the following equation
\begin{equation}
    z'(t)=Az(t)+Bz(t-\tau),
    \label{Ec.reducida.test.EDPR}
\end{equation}
where $z(t)=\big(z_1(t),z_2(t)\big)$, and the matrices are defined as
\begin{equation*}
    A=\begin{pmatrix}
        a_1 T&0 \\
        0&a_2 T
    \end{pmatrix},
    \quad
    B=e^{l\frac{\pi}{2}}\begin{pmatrix}
        -I&\big( l+\frac{\pi^2}{4} \big)I \\
        -\big(l+\frac{\pi^2}{4}\big) I& -I
    \end{pmatrix},
\end{equation*}
being $I=I_{n-1}$ the $n-1$-dimensional identity matrix and $T$ is
\begin{equation*}
    T=\frac{1}{(\Delta x)^2}\begin{pmatrix}
        -2&1&0&0&\dots&0\\
        1&-2&1&0&\dots&0\\
        0&1&-2&\ddots&&\vdots \\
        \vdots&&\ddots&\ddots&\ddots&0\\
        \vdots&&&\ddots&-2&1\\
        0&\dots&\dots&0&1&-2
    \end{pmatrix}_{(n-1)(n-1)},
\end{equation*}
which is also called a \textit{Toeplitz} matrix. For a generic Toeplitz matrix 
\begin{equation*}
    M=\begin{pmatrix}
        a&b&0&0&\dots&0\\
        c&a&b&0&\dots&0\\
        0&c&a&\ddots&&\vdots \\
        \vdots&&\ddots&\ddots&\ddots&0\\
        \vdots&&&\ddots&a&b\\
        0&\dots&\dots&0&c&a
    \end{pmatrix}_{N\cdot N},
\end{equation*}
the eigenvalues $\bar{\lambda_i}$ are known and have the following form
\begin{equation*}
    \bar{\lambda_i}=a-2\sqrt{bc}\cos\Big( \frac{i\pi}{N+1}\Big),\quad i=1,\dots,N. 
\end{equation*}

Hence, the matrix A will be negative definite, as well as Hermitian, provided that $a_1,a_2>0$ is verified.

By applying \ref{b0_family} with $\beta_0=3/5$ to \eqref{Ec.reducida.test.EDPR} we obtain:
\begin{equation*}
    z_{n+2}-\frac{4}{5}z_{n+1}-\frac{1}{5}z_n=h\left(Az_{n+2}-\frac{2}{5}Az_{n+1}+\frac{3}{5}Az_n+Bz_{n+2-m}-\frac{2}{5}Bz_{n+1-m}+\frac{3}{5}Bz_{n-m} \right).
\end{equation*}

This solution vector $z(t)$ would be $(n-1)+(n-1)=2n-2$ dimensional, with its first half of components corresponding to $u_1$ and the second half to $u_2$, excluding the boundary nodes $x_0=-1,x_n=1$. 

The initial and boundary conditions will be set so that the solutions to our system are the following:
\begin{equation}
    \left\{
    \begin{array}{cc}
         u_1(t,x)=\phi_1(t,x)=e^{lt}\sin(t)\cos(\frac{\pi}{2}x),  \\
          u_2(t,x)=\phi_2(t,x)=e^{lt}\cos(t)\cos(\frac{\pi}{2}x).
    \end{array}
    \right.
    \label{sol_EDPR_IMEXcond2}
\end{equation}
As a consequence, the sign of the parameter $l$ will determine the behavior of the solutions; i.e., divergent solutions would appear for $l>0$ and damped solutions for $l<0$. The unconditional stability of our method ensures that the numerical solutions will inherit these characteristics.

By setting the parameters in the following way $a_1=1,\;a_2=1,\;l=-1/5,\;n=100$, it is possible to verify the following: the matrices $A,B$ commute, all the eigenvalues of $A$ have multiplicity 2 (due to $a_1=a_2$), and there are only two distinct eigenvalues of $B$, both with multiplicity $n-1=99$. 

\begin{figure}[h]
    \centering 
    \begin{subfigure}[b]{0.49\textwidth}
        \centering
        \includegraphics[width=\textwidth]{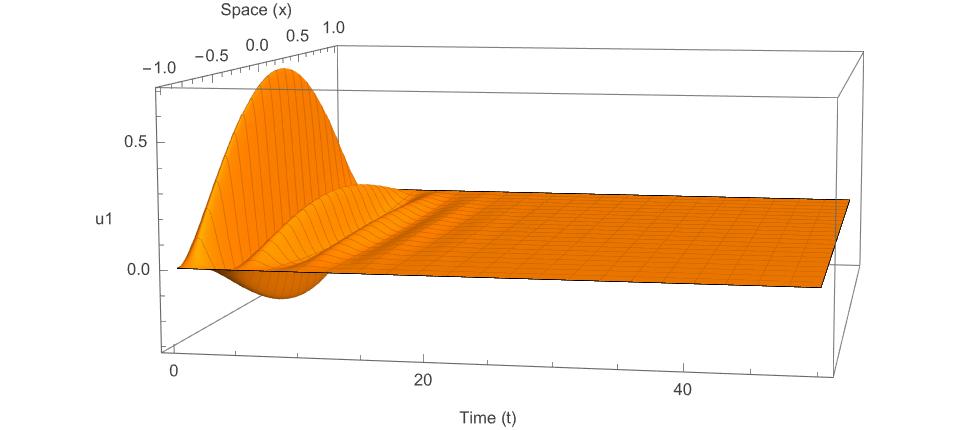}
    \end{subfigure}
    \hfill
    \begin{subfigure}[b]{0.49\textwidth}
        \centering
        \includegraphics[width=\textwidth]{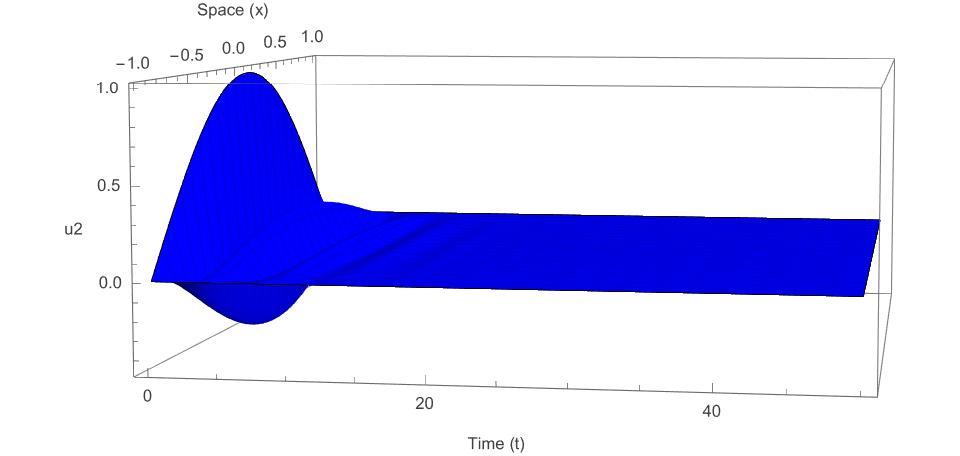}
    \end{subfigure}
    \caption{Numerical solutions $u_1$ (left) and $u_2$ (right) for $h=0.1$ and $t_f=50$.}
    \label{fig:sol_u1_u2_syst_PDDE} 
\end{figure}
In Figs. \ref{fig:sol_u1_u2_syst_PDDE} we can observe the damped oscillation behavior from both functions, directly related to the negative value of $l$.

To illustrate the utility of the concept of the Field of Values in problems like the one we are dealing with, Figs. \ref{fig:fov_pdde_2A_l-02} and \ref{fig:fov_pdde_2A_l02} show the FOV of $A^{-1}B$ for two cases of $l$. 
By using the result given in \textbf{Corollary} \ref{corol_FOV_stability}, we can see that in the asymptotically stable case ($l<0$), the numerical solutions preserve this property. In contrast, in the divergent case, the FOV of $A^{-1}B$ is not even contained in the unit disk.

\begin{figure}[H]
    \centering 
    \begin{subfigure}[b]{0.3\textwidth}
        \centering
        \includegraphics[width=\textwidth]{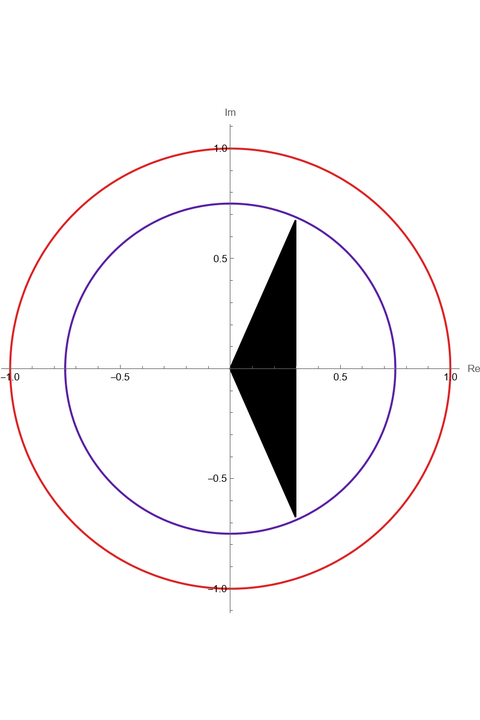}
        \caption{$l=-1/5$}
        \label{fig:fov_pdde_2A_l-02}
    \end{subfigure}
    \qquad \qquad
    \begin{subfigure}[b]{0.25\textwidth}
        \centering
        \includegraphics[width=\textwidth]{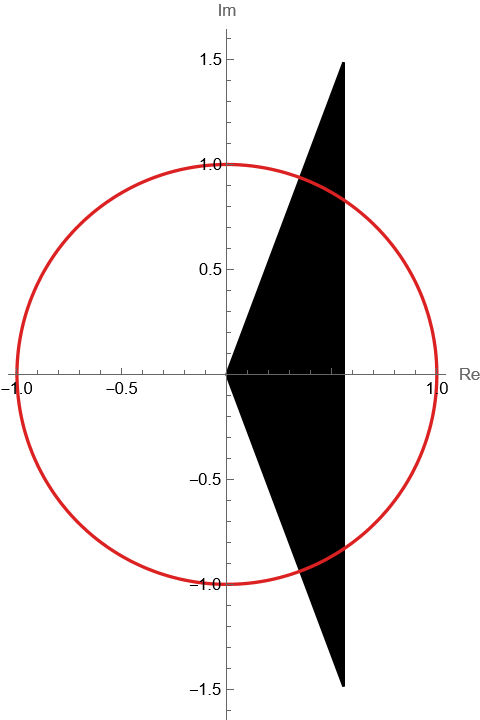}
        \caption{$l=1/5$}
        \label{fig:fov_pdde_2A_l02}
    \end{subfigure}

    \caption{Field of Values $F(A^{-1}B)$ for $n=100$ and some values of $l$. The red circle has radius 1, while the purple one has radius 0.75.}
    \label{fig:fov_pdde_2A} 
\end{figure}

\section{Conclusions and future challenges}
\label{sec:6}
The main focus of this paper was to study the stability of a novel $\beta_0$-family of IMEX multistep methods \eqref{b0_family} for the test problem $y'(t) = Ay(t) + By(t-\tau(t,y(t)))$, without assuming simultaneous diagonalizability of the matrices. By utilizing the field of values (FOV) of the matrix $A^{-1}B$, we provided a rigorous theoretical framework for asymptotic stability analysis.

For grid-aligned constant delays ($u = 0$), we have demonstrated via the argument principle that this family of second-order multistep methods is unconditionally stable, provided $A$ is Hermitian negative-definite and the FOV of $A^{-1}B$ lies within the complex unit disk. Furthermore, the applicability of this theoretical framework has been successfully extended to semi-discretized parabolic partial delay differential equations (PDDEs), as illustrated by the Fisher-Kolmogorov-Petrovsky-Piskunov (FKPP) model. In particular, the latter model has been used in population dynamics in areas such as epidemiology, ecology, or chemistry.

Additionally, another key contribution of this work is the introduction of time- or state-dependent delays, which alters the stability landscape. The necessity of using a continuous extension (specifically, a second-order Lagrange interpolant for off-grid delayed nodes, $u \neq 0$ ) introduces a causality condition.  For a constant delay whose delayed arguments do not coincide with the
grid, the quadratic interpolation symbol is nonamplifying on the unit
circle. Under the fixed-offset and no-overlap assumptions, the
interpolation therefore does not reduce the scalar stability disk.

For genuinely time- or state-dependent delays, the discrete problem is
nonautonomous because the delay index and interpolation weights may
vary at every step.   A complete variable-delay
result, therefore, requires a common Lyapunov estimate or another
product-level condition.

Natural directions for future work include common or
parameter-dependent Lyapunov certificates for the variable-delay
examples, nonlinear invariant-region estimates for nonlinear problems,
and extensions to non-Hermitian dissipative operators and higher-order
formulas.

\section*{Acknowledgments}
   The authors would like to thank  Prof. Giovanni Pagano and Gianluca Frasca Caccia,  for fruitful conversations.  This research was partly funded by  Fundaci\'on Sol\'orzano through the project FS/11-2024. 
\vspace{0.1cm}

\noindent
\small
\textbf{Data Availability} Data sharing is not applicable to this article as no datasets were generated or analyzed during the current study.

\noindent
\textbf{Code Availability} Selected code files used in this study (specifically, the symbolic computation of the LTE, the test for variable delays, and the FKPP problem) are publicly available on GitHub at \url{https://github.com/javiroguti/DDEs-IMEX-stability}. The remaining code is available from the corresponding author upon reasonable request.

\section*{Declarations}
\small
\textbf{Conflict of interest}
The authors declare that they have no conflict of interest.

\bibliographystyle{unsrt}
\bibliography{refs}

\end{document}